\documentclass[11pt]{article}

\usepackage{amsthm}
\usepackage{amssymb}
\usepackage{amsmath}
\usepackage{comment}
\usepackage{thm-restate}
\usepackage{url} 
\usepackage{hyperref}
\usepackage[noabbrev,capitalise]{cleveref}

\usepackage[utf8]{inputenc}

\usepackage{color}
\usepackage[normalem]{ulem}

\usepackage[dvipsnames, table]{xcolor}

\usepackage[margin=.75in]{geometry}

\newcommand{\mc}[1]{\mathcal{#1}}

\renewcommand{\vert}{\textup{\textsf{v}}}
\newcommand{\e}{\textup{\textsf{e}}}
\renewcommand{\d}{\textup{\textsf{d}}}

\theoremstyle{plain}
\newtheorem{thm}{Theorem}[section]
\newtheorem{lem}[thm]{Lemma}
\newtheorem{claim}{Claim}[thm]

\newtheorem{cor}[thm]{Corollary}

\newtheorem{conj}[thm]{Conjecture}

{\noindent \emph{Proof.} {}{#1}{}}{\hfill
	$\Diamond$\vspace{1em}}

\theoremstyle{plain} 
\newcommand{\thistheoremname}{}
\newtheorem{genericthm}{\thistheoremname}

\theoremstyle{definition}
\newtheorem{definition}[thm]{Definition}

\title{Reducing Linear Hadwiger's Conjecture to Coloring Small Graphs}

\author{
Michelle Delcourt
\thanks{Department of Mathematics, Toronto Metropolitan University,
Toronto, Ontario M5B 2K3, Canada {\tt mdelcourt@torontomu.ca}. Research supported by NSERC under Discovery Grant No. 2019-04269.}
\and
Luke Postle
\thanks{Combinatorics and Optimization Department,
University of Waterloo, Waterloo, Ontario N2L 3G1, Canada {\tt lpostle@uwaterloo.ca}. Partially supported by NSERC
under Discovery Grant No. 2019-04304.}}
\date{\today}

\begin{document}

\maketitle

\begin{center}
	\emph{Dedicated to the memory of Robin Thomas}
\end{center}

\begin{abstract} 
	In 1943, Hadwiger conjectured that every graph with no $K_t$ minor is $(t-1)$-colorable for every $t\ge 1$. In the 1980s, Kostochka and Thomason independently proved that every graph with no $K_t$ minor has average degree $O(t\sqrt{\log t})$ and hence is $O(t\sqrt{\log t})$-colorable.  Recently, Norin, Song and the second author showed that every graph with no $K_t$ minor is $O(t(\log t)^{\beta})$-colorable for every $\beta > 1/4$, making the first improvement on the order of magnitude of the $O(t\sqrt{\log t})$ bound. The first main result of this paper is that every graph with no $K_t$ minor is $O(t\log\log t)$-colorable.
	
This is actually a corollary of our main technical result that the chromatic number of a $K_t$-minor-free graph is bounded by $O(t(1+f(G,t)))$ where $f(G,t)$ is the maximum of $\frac{\chi(H)}{a}$ over all $a\ge \frac{t}{\sqrt{\log t}}$ and $K_a$-minor-free subgraphs $H$ of $G$ that are small (i.e. $O(a\log^4 a)$ vertices). This has a number of interesting corollaries. First as mentioned, using the current best-known bounds on coloring small $K_t$-minor-free graphs, we show that $K_t$-minor-free graphs are $O(t\log\log t)$-colorable. Second, it shows that proving Linear Hadwiger's Conjecture (that $K_t$-minor-free graphs are $O(t)$-colorable) reduces to proving it for small graphs. Third, we prove that $K_t$-minor-free graphs with clique number at most $\sqrt{\log t}/ (\log \log t)^2$ are $O(t)$-colorable. This implies our final corollary that Linear Hadwiger's Conjecture holds for $K_r$-free graphs for every fixed $r$; more generally, we show there exists $C\ge 1$ such that for every $r\ge 1$, there exists $t_r$ such that for all $t\ge t_r$, every $K_r$-free $K_t$-minor-free graph is $Ct$-colorable.	
	
One key to proving the main theorem is building the minor in two new ways according to whether the chromatic number `separates' in the graph: sequentially if the graph is  the chromatic-inseparable and recursively if the graph is chromatic-separable. The other key is a new standalone result that every $K_t$-minor-free graph of average degree $d=\Omega(t)$ has a subgraph on $O(t \log^3 t)$ vertices with average degree $\Omega(d)$. 
\end{abstract}

\section{Introduction}

All graphs in this paper are finite and simple. Given graphs $H$ and $G$, we say that $G$ has an \emph{$H$ minor} if a graph isomorphic to $H$ can be obtained from a subgraph of $G$ by contracting edges. We denote the complete graph on $t$ vertices by $K_t$. For the benefit of the reader, we collect other terms and notation used throughout the paper in Section~\ref{sub:not}.

In 1943 Hadwiger made the following famous conjecture.

\begin{conj}[Hadwiger's Conjecture~\cite{Had43}]\label{Hadwiger} For every integer $t \geq 1$, every graph with no $K_{t}$ minor is $(t-1)$-colorable. 
\end{conj}

Hadwiger's Conjecture is widely considered among the most important problems in graph theory and has motivated numerous developments in graph coloring and graph minor theory. For an overview of major progress on Hadwiger's Conjecture, we refer the reader to~\cite{NPS19}, and to the recent survey by Seymour~\cite{Sey16Survey} for further background.

The following is a natural weakening of Hadwiger's Conjecture.

\begin{conj}[Linear Hadwiger's Conjecture~\cite{ReeSey98,Kaw07, KawMoh06}]\label{c:LinHadwiger} There exists a constant $C>0$ such that for every integer $t \geq 1$, every graph with no $K_{t}$ minor is $Ct$-colorable. 
\end{conj}

For many decades, the best general bound on the number of colors needed to properly color every graph with no $K_t$ minor had been $O(t\sqrt{\log{t}})$, a result obtained independently by Kostochka~\cite{Kostochka82,Kostochka84} and Thomason~\cite{Thomason84} in the 1980s. The results of \cite{Kostochka82,Kostochka84,Thomason84} bound the ``degeneracy" of graphs with no $K_t$ minor. Recall that a graph $G$ is \emph{$d$-degenerate} if every non-empty subgraph of $G$ contains a vertex of degree at most $d$. A standard inductive argument shows that every $d$-degenerate graph is $(d+1)$-colorable. Thus the following bound on the degeneracy of graphs with no $K_t$ minor gives a corresponding bound on their chromatic number.

\begin{thm}[\cite{Kostochka82,Kostochka84,Thomason84}]\label{t:KT} Every graph with no $K_t$ minor is $O(t\sqrt{\log{t}})$-degenerate.
\end{thm}

Kostochka~\cite{Kostochka82,Kostochka84} and de la Vega~\cite{Vega83} have shown that there exist graphs with no $K_t$ minor and minimum degree $\Omega(t\sqrt{\log{t}})$. Thus the bound in \cref{t:KT} is tight. Until very recently $O(t\sqrt{\log{t}})$ remained the best general bound for the chromatic number of graphs with no $K_t$ minor when Norin, Song and the second author~\cite{NPS19} improved this with the following theorem.

\begin{thm}[\cite{NPS19}]\label{t:ordinaryHadwiger}
For every $\beta > \frac 1 4$, every graph with no $K_t$ minor is $O(t (\log t)^{\beta})$-colorable.
\end{thm}

The first main result of this paper is the following improvement to Theorem~\ref{t:ordinaryHadwiger}. 

\begin{thm}\label{thm:ordinaryHadwiger3}
Every graph with no $K_t$ minor is $O(t \log \log t)$-colorable. 
\end{thm}

Theorem~\ref{thm:ordinaryHadwiger3} is in fact a corollary of our more technical main result, Theorem~\ref{t:tech}, which is as follows. We note that Theorem~\ref{t:tech} does not rely on Theorem~\ref{t:ordinaryHadwiger} or the results in~\cite{NPS19} and hence the proof presented here is independent of the proof in that paper.

\begin{restatable}{thm}{Tech}\label{t:tech} There exists an integer $C=C_{\ref{t:tech}} \ge 1$ such that the following holds: Let $t\ge 3$ be an integer. Let $G$ be a graph and let 
$$f_{\ref{t:tech}}(G,t) := \max_{H\subseteq G} \left\{ \frac{\chi(H)}{a}:~t\ge a \ge \frac{t}{\sqrt{\log t}},~\vert(H)\le C a \log^4 a,~H \text{ is $K_a$-minor-free }\right\}.$$ 
If $G$ has no $K_t$ minor, then
$$\chi(G) \le C\cdot t \cdot (1+f_{\ref{t:tech}}(G,t)).$$ 
\end{restatable}

Theorem~\ref{t:tech} has a number of interesting corollaries. As mentioned, our first corollary of Theorem~\ref{t:tech} is Theorem~\ref{thm:ordinaryHadwiger3}. This follows straightforwardly by using the best-known bounds on the chromatic number of small $K_t$-minor-free graphs. We derive Theorem~\ref{thm:ordinaryHadwiger3} in Section~\ref{s:small}. 

Our second corollary is that Linear Hadwiger's Conjecture reduces to small graphs as follows.

\begin{cor}\label{cor:ReduceToSmall}
There exists an integer $C=C_{\ref{cor:ReduceToSmall}} \ge 1$ such that the following holds: If for every integer $t\ge 3$ we have that every $K_t$-minor-free graph $H$ with $\vert(H)\le Ct \log^4 t$ satisfies $\chi(H)\le Ct$, then for every integer $t\ge 3$ we have that every $K_t$-minor-free graph $G$ satisfies $\chi(G) \le C^2 t$.
\end{cor}
\begin{proof}
Follows from Theorem~\ref{t:tech} by setting $C=2C_{\ref{t:tech}}$.
\end{proof}

Our third corollary of Theorem~\ref{t:tech} shows that Linear Hadwiger's Conjecture holds if the clique number of the graph is small as a function of $t$.

\begin{cor}\label{cor:HadwigerSubLogClique}
There exists $C=C_{\ref{cor:HadwigerSubLogClique}}\ge 1$ such that the following holds: Let $t\ge 3$ be an integer. If $G$ is a $K_t$-minor-free graph with $\omega(G) \le \frac{\sqrt{\log t}}{(\log\log t)^2}$, then $\chi(G)\le Ct$.
\end{cor}

We derive Corollary~\ref{cor:HadwigerSubLogClique} in Section~\ref{s:small} by utilizing better bounds on the chromatic number of graphs with small clique number. 

In 2003, K\"uhn and Osthus~\cite{KO03} proved that Hadwiger's Conjecture holds for graphs of girth at least five provided that $t$ is sufficiently large. In 2005, K\"uhn and Osthus~\cite{KO05} extended this result to the class of $K_{s,s}$-free graphs for any fixed positive integer $s\ge 2$. Along this line, as a corollary of Theorem~\ref{t:tech}, we prove that Linear Hadwiger's Conjecture holds for the class of $K_r$-free graphs for every fixed $r$. More generally, the following holds.

\begin{cor}\label{cor:ordinaryHadwigerKrFree}
There exists $C=C_{\ref{cor:ordinaryHadwigerKrFree}}\ge 1$ such that for every integer $r\ge 3$, there exists $t_r > 0$ such that for all integers $t\ge t_r$, every $K_r$-free $K_t$-minor-free graph is $Ct$-colorable.
\end{cor}
\begin{proof}
Let $C=C_{\ref{cor:HadwigerSubLogClique}}$. For each integer $r\ge 3$, let $t_r$ be such that $r\le \frac{\sqrt{\log t_r}}{(\log\log t_r)^2}$. By Corollary~\ref{cor:HadwigerSubLogClique}, it follows that for all integers $t\ge t_r$ we have that every $K_r$-free $K_t$-minor-free is $Ct$-colorable.
\end{proof}

Note that the constant in Corollary~\ref{cor:ordinaryHadwigerKrFree} does not depend on $r$ but requires $t$ to be sufficiently large with respect to $r$. Allowing the constant to depend on $r$ eliminates that latter assumption as follows.

\begin{cor}\label{cor:ordinaryHadwigerKrFree2}
For every integer $r\ge 3$, there exists $C_r\ge 1$ such that for all integers $t\ge 1$, every $K_r$-free $K_t$-minor-free graph is $C_r \cdot t$-colorable.
\end{cor}
\begin{proof}
Let $t_r$ be as in Corollary~\ref{cor:ordinaryHadwigerKrFree} and let $C_r = \max\{ C_{\ref{cor:ordinaryHadwigerKrFree}},~ 2\cdot C_{\ref{t:tech}}^2 \cdot \log^4 t_r\}$. By Corollary~\ref{cor:ordinaryHadwigerKrFree}, we have that for all integers $t\ge t_r$, every $K_r$-free $K_t$-minor-free is $C_{\ref{cor:ordinaryHadwigerKrFree}}\cdot t$-colorable and hence is $C_r\cdot t$-colorable. By Theorem~\ref{t:tech}, we have that for all integers $t$ with $t_r\ge t\ge 1$, every $K_r$-free $K_t$-minor-free is $2\cdot C_{\ref{t:tech}}^2 \cdot (\log^4 t_r)\cdot t$-colorable and hence is $C_r\cdot t$-colorable. Hence for all integers $t\ge 1$, we find that every $K_r$-free $K_t$-minor-free graph is $C_r\cdot t$-colorable.
\end{proof}

Thus Corollary~\ref{cor:ordinaryHadwigerKrFree2} proves Linear Hadwiger's for the class of $K_r$-free graphs for every fixed $r$. On the other hand, in 2017, Dvo\v{r}\'ak and Kawarabayashi~\cite{DK17} showed that there exist triangle-free graphs of tree-width at most $t$ and chromatic number at least $\left\lceil \frac{t+3}{2} \right\rceil$. Hence the result in Corollary~\ref{cor:ordinaryHadwigerKrFree2} is tight up to the multiplicative constant.


\subsection{Notation}\label{sub:not}

We use largely standard graph-theoretical notation. We denote by $\vert(G)$ and $\e(G)$ the number of vertices and edges of a graph $G$, respectively, and denote by $\d(G)=\e(G)/\vert(G)$ the \emph{density} of a non-empty graph $G$. We use $\chi(G)$ to denote the chromatic number of $G$, $\alpha(G)$ to denote the independence number of $G$, $\omega(G)$ to denote the clique number of $G$, $\delta(G)$ to denote the minimum degree of $G$, and $\kappa(G)$ to denote the (vertex) connectivity of $G$. 
 
The degree of a vertex $v$ in a graph $G$ is denoted by $\deg_G(v)$ or simply by $\deg(v)$ if there is no danger of confusion. We denote by $G[X]$ the subgraph of $G$ induced by a set $X \subseteq V(G)$. If $A$ and $B$ are disjoint subsets of $V(G)$, then we let $G(A,B)$ denote the bipartite subgraph with $V(G(A,B))=A\cup B$ and $E(G(A,B)) = \{uv\in E(G): u\in A,~v\in B\}$.

We say that vertex-disjoint subgraphs $H$ and $H'$ of a graph $G$ are \emph{adjacent} if there exists an edge of $G$ with one end in $V(H)$ and the other end in $V(H')$, and we say that $H$ and $H'$ are \emph{non-adjacent}, otherwise. For a positive integer $n$, let $[n]$ denote $\{1,2,\ldots,n\}$. A collection $\mc{X} = \{X_1,X_2,\ldots,X_h\}$ of pairwise vertex-disjoint subgraphs of $V(G)$ is a \emph{model of a graph $H$ in a graph $G$} if $X_i$ is connected for every $i \in [h]$, and there exists a bijection $\phi: V(H) \to [h]$ such that $X_{\phi(u)}$ and $X_{\phi(v)}$ are adjacent for every $uv \in E(H)$. It is well-known and not hard to see that $G$ has an $H$ minor if and only if there exists a model of $H$ in $G$. We write $V(\mathcal{X})$ for $\bigcup_{i\in [h]} V(X_i)$. We say that a model $\mc{X}$ as above is \emph{rooted at $S$} for $S \subseteq V(G)$ if $|S|=h$  and  $|V(X_i) \cap S|=1$ for every $i \in [h]$. The logarithms in this paper are natural unless specified otherwise.  

\section{Outline of Proof}\label{sec:outline}

The proof of Theorem~\ref{t:ordinaryHadwiger} in~\cite{NPS19} proceeds by proving the contrapositive and consists of three parts:
\begin{enumerate}
\item[(1)] Show that small $K_t$-minor-free graphs are $O(t\log \log t)$-colorable where small means the graph has at most $t (\log t)^{O(1)}$ vertices.
\item[(2)] Show that a dense $K_t$-minor-free graph has a small dense highly-connected subgraph; and hence in combination with (1), a high chromatic $K_t$-minor-free graph has many vertex-disjoint dense highly-connected subgraphs.
\item[(3)] Use connectivity tools to build a $K_t$ minor \emph{all at once} by building smaller minors inside each subgraph and linking them appropriately.
\end{enumerate}

Our proof of Theorem~\ref{t:tech} follows this simplified strategy in outline; however for parts (2) and (3), we use new approaches for the proof. In particular for (3), we build the minor in two new ways: sequentially and recursively. Let us describe the differences in more detail.

For (1), the best bound is still $O(t\log \log t)$ colors and thus, using the results of this paper, will be the bottleneck for Linear Hadwiger's Conjecture; however, in Section~\ref{s:small} we improve the bound for small $K_t$ minor-free graphs of small clique number (as in Corollary~\ref{cor:HadwigerSubLogClique}) to $O(t)$ colors using state of the art coloring techniques.

For (2), Norin, Song and the second author~\cite{NPS19} proved a so-called ``density increment theorem" which finds either a denser minor or a small dense subgraph as follows.

\begin{thm}\label{t:newforced} Let $G$ be a graph with $\d(G) \ge 1$, and let $D > 0$ be a constant. Let $s=D/\d(G)$ and let $g_{\ref{t:newforced}}(s) := s^{o(1)}$.  Then $G$ contains at least one of the following: 
\begin{description}
		\item[(i)] a minor $J$ with $\d(J) \geq D$, or
		\item[(ii)] a subgraph $H$ with $\vert(H) \leq g_{\ref{t:newforced}}(s) \cdot \frac{D^2}{\d(G)}$ and $\d(H) \geq \frac{\d(G)}{g_{\ref{t:newforced}}(s)}$.
	\end{description}  
\end{thm}

Since graphs with density $\Omega(t \sqrt{\log t})$ have $K_t$ minors by Theorem~\ref{t:KT}, the above theorem with $s=O(\sqrt{\log t})$ yields the following result as a corollary.

\begin{thm}\label{t:SmallConn}
Let $t\ge 3$ be an integer and define $f_{\ref{t:SmallConn}}(t) :=(\log t)^{o(1)}$. For every integer $k\ge t$, if $G$ is a graph with $\d(G) \ge k \cdot f_{\ref{t:SmallConn}}(t)$ and $G$ contains no $K_t$ minor, then $G$ contains a $k$-connected subgraph $H$ with $\vert(H) \le t \cdot f_{\ref{t:SmallConn}}(t) \cdot \log t$.
\end{thm}

While the function in Theorem~\ref{t:newforced} can be improved, even to $O\left((\log s)^6\right)$ using similar but more technical methods, its use still presented a bottleneck in reducing Linear Hadwiger's Conjecture to small graphs. Thus instead of improving Theorem~\ref{t:newforced} directly, we instead improve Theorem~\ref{t:SmallConn} directly via a completely different and shorter approach. 

In the proof of Theorem~\ref{t:newforced} and any attempted improvement, the denser minor from condition (i) is built iteratively (over roughly $\log s$ iterations); in each stage, only components of small size (at most $\log s$) are allowed to be contracted while showing that not too many edges are lost. If the process fails, a small dense subgraph as in condition (ii) is found. The small size of contractions (and hence the iterative approach) seemed necessary so as to limit the error in the size and density of the small dense subgraph. Unfortunately, this iterative approach then seems to have a natural barrier for $g_{\ref{t:newforced}}(s)$ of around $\log s$ and hence even in the best case would still present an obstacle to proving Linear Hadwiger's for all graphs let alone for the special classes in Corollary~\ref{cor:HadwigerSubLogClique}.

Thus we use a new approach where we attempt to find a $K_t$ minor in only one stage and when that fails, we find a small dense subgraph. Namely, in Section~\ref{s:smallConn}, we prove the following theorem which proves the existence in a $K_t$-minor-free graph $G$ of a small dense subgraph with only a constant loss in density. 

\begin{restatable}{thm}{SmallConnNew}\label{t:SmallConn2}
There exists an integer $C=C_{\ref{t:SmallConn2}} \ge 1$ such that the following holds: Let $t\ge 1$ be an integer. For every integer $k\ge t$, if $G$ is a graph with $\d(G) \ge Ck$ and $G$ contains no $K_t$ minor, then $G$ contains a non-empty $k$-connected subgraph $H$ with $\vert(H) \le C^2 \cdot t \cdot \log^3 t$.
\end{restatable}

For (3), the $K_t$ minor in~\cite{NPS19} is built all at once which naturally leads to the $(\log t)^{1/4}$ term in Theorem~\ref{t:ordinaryHadwiger}, which is the square root of the $\sqrt{\log t}$ factor in Theorem~\ref{t:KT}. Thus we turn to the hardest part of the proof of our main result, Theorem~\ref{t:tech}, where we build the minor in two completely new ways, sequentially and recursively. 

The utilized construction method splits according to two main cases as determined by the following key definition.

\begin{definition}
Let $s$ be a nonnegative integer. We say that a graph $G$ is \emph{$s$-chromatic-separable} if there exist two vertex-disjoint subgraphs $H_1,H_2$ of $G$ such that $\chi(H_i)\ge \chi(G)-s$ for each $i\in \{1,2\}$ and that $G$ is \emph{$s$-chromatic-inseparable} otherwise. 
\end{definition}

In Section~\ref{s:inseparable}, we prove the following lemma that covers the chromatic-inseparable case. 

\begin{restatable}{lem}{Inseparable}\label{lem:inseparable2}
There exists an integer $C=C_{\ref{lem:inseparable2}} \ge 1$ such that the following holds: Let $t\ge 3$ be an integer. Let $G$ be a graph and let
$$g_{\ref{lem:inseparable2}}(G,t) := \max_{H\subseteq G} \left\{ \frac{\chi(H)}{t}: \vert(H)\le C t \log^4 t,~H \text{ is $K_t$-minor-free }\right\}.$$ If $G$ is $Ct\cdot(1+g_{\ref{lem:inseparable2}}(G,t))$-chromatic-inseparable and $\chi(G)\ge 2\cdot Ct\cdot (1+g_{\ref{lem:inseparable2}}(G,t))$, then $G$ contains a $K_t$ minor.
\end{restatable}

The minor in Lemma~\ref{lem:inseparable2} is built sequentially over $\lceil \log t \rceil$ steps; in each step a $K_{\lceil \frac{t}{\log t} \rceil, t}$ minor is built and then linked to the previously built minors but also via chromatic inseparability to a high chromatic, highly-connected subgraph yet unused for building the remainder of the $K_{t}$ minor. 

As for the chromatic-separable case, the $K_t$ minor is built recursively, namely by trinary recursion with a recursion depth of $O(\log \log t)$. In each level except the last, a $K_s$ minor is built by finding three vertex-disjoint high chromatic, highly-connected subgraphs, recursively constructing a $K_{2s/3}$ minor in each subgraph and then linking the three $K_{2s/3}$ minors together. The existence of three such subgraphs is guaranteed only by means of chromatic separability. As for the last level, the existence of a $K_{t/\log t}$ minor follows from Theorem~\ref{t:KT}.

The above approach, however, would assume that high-chromatic subgraphs are always chromatic-separable. Keeping the cases separate as such would then unfortunately lead to an additional $\log \log t$ multiplicative factor. To avoid this factor and prove Theorem~\ref{t:tech}, we instead use Lemma~\ref{lem:inseparable2} as a black box while using a clever trick to combine the two cases (chromatic-separable and chromatic-inseparable) into one general case (instead of doing the `always chromatic-separable' case separately). The key to overcoming this technicality is to ensure the desired linkage of ancestor nodes before invoking the chromatic-separability. We discuss this approach in more detail in Section~\ref{s:separable} before proving Theorem~\ref{t:tech} in that section.

\subsection{Outline of Paper}

In Section~\ref{s:small}, we prove results about the chromatic number of small $K_t$-minor-free graphs. In Section~\ref{s:smallConn}, we prove Theorem~\ref{t:SmallConn2}. In Section~\ref{s:prelim}, we collect a toolkit of connectivity results needed in the proofs of Lemma~\ref{lem:inseparable2} and Theorem~\ref{t:tech}. In Section~\ref{s:inseparable}, we prove Lemma~\ref{lem:inseparable2}. In Section~\ref{s:separable}, we prove Theorem~\ref{t:tech}.

\section{Coloring Small Graphs}\label{s:small}

To apply Theorem~\ref{t:tech}, we need a bound on the chromatic number of very small graphs with no $K_t$-minor. Useful for that purpose is the \emph{Hall ratio} of a graph $G$ defined to be 
$$\rho(G):=\max\left\{ \frac{\vert(H)}{\alpha(H)} : H \text{ is a non-empty subgraph of }G\right\}.$$

Note $\vert(G) \ge \rho(G)$ for every graph $G$ as $\alpha(G)\ge 1$ for every non-empty graph $G$. Here then is a useful lemma that upper bounds the chromatic number of a graph in terms of $\vert(G)$ and $\rho(G)$.

\begin{lem}\label{lem:small}
If $G$ is a non-empty graph, then 
$$\chi(G) \leq \left(2+\log \left(\frac{\vert(G)}{\rho(G)}\right)\right)\rho(G).$$
\end{lem}
\begin{proof}
Suppose not. Let $H$ be an induced non-empty subgraph of $G$ with $\vert(H)\ge \rho(G)$ such that
$$\chi(H) > \left(2+\log \left(\frac{\vert(H)}{\rho(G)}\right)\right)\rho(G),$$
and subject to that $\vert(H)$ is minimized. Note that $H$ exists since $G$ satisfies the equation and $\vert(G)\ge \rho(G)$.  By definition of Hall ratio, $\alpha(H)\ge \frac{\vert(H)}{\rho(G)}$. Hence there exists an independent set $I$ of $H$ with $|I| \ge \frac{\vert(H)}{\rho(G)}$. Let $H' = H\setminus I$. Note that $\chi(H)\le \chi(H')+1$. Moreover since $1-x \le e^{-x}$ for every $x$, we have that 
$$\vert(H') \le \left(1-\frac{1}{\rho(G)}\right) \vert(H) \le e^{-\frac{1}{\rho(G)}} \cdot \vert(H).$$ 

First suppose $\vert(H')\ge \rho(G)$. Then by the minimality of $H$, 
$$\chi(H') \leq \left(2+\log \left(\frac{\vert(H')}{\rho(G)}\right)\right)\rho(G) \le \left(2+\log \left(\frac{\vert(H)}{\rho(G)}\right) - \frac{1}{\rho(G)}\right)\rho(G) < \chi(H)-1,$$
a contradiction.

So we may assume that $\vert(H') < \rho(G)$. But then $\chi(H') \le \vert(H') < \rho(G)$. Hence $\chi(H) \le \rho(G)+1$. Since $\rho(G) \ge \rho(K_1)=1$, we find that $\chi(H)\le 2\rho(G)$, a contradiction since $2 \le 2 + \log \left(\frac{\vert(H)}{\rho(G)}\right)$ as $\vert(H)\ge \rho(G)$.
\end{proof}

\begin{cor}\label{cor:smallbound}
If $G$ is a non-empty graph and $p$ is a positive integer such that $\rho(G)\le p$, then 
$$\chi(G) \leq \left(2+\max\left\{\log \left(\frac{\vert(G)}{p}\right),~0\right\}\right)\cdot p.$$
\end{cor}
\begin{proof}
First suppose $p \ge \vert(G)$. Then $\chi(G)\le \vert(G)\le p \le 2p$ as desired. So we assume that $p < \vert(G)$. Let $f(x):= \left(2+\log \left(\frac{\vert(G)}{x}\right)\right)\cdot x$. By Lemma~\ref{lem:small}, $\chi(G) \le f(\rho(G))$. Furthermore, we note that $f'(x) = 1+\log\left(\frac{\vert(G)}{x}\right)$. Hence $f'(x)\ge 0$ for all $x\in [\rho(G),~\vert(G)]$ and hence $f$ is increasing in the interval $[\rho(G),~\vert(G)]$. Since $p\in [\rho(G),~\vert(G)]$, we find that $f(\rho(G))\le f(p)$. Hence $\chi(G)\le f(p)$ as desired. 
\end{proof}

A classical theorem of Duchet and Meyniel~\cite{DucMey82} from 1982 bounds the Hall ratio of $K_t$-minor-free graphs as follows.

\begin{thm}[\cite{DucMey82}]\label{t:DucMey}
If $G$ is a $K_t$-minor-free graph, then $\rho(G) \le 2(t-1)$.
\end{thm}

We note that Fox~\cite{Fox10} in 2010 was the first to improve the multiplicative factor of $2$ in Theorem~\ref{t:DucMey}, while building on those techniques, the best current bound to date is due to Balogh and Kostochka~\cite{BK11} from 2011.

We are now ready to derive Theorem~\ref{thm:ordinaryHadwiger3}.

\begin{proof}[Proof of Theorem~\ref{thm:ordinaryHadwiger3}.]
Let $C:= 15 \cdot C_{\ref{t:tech}}^2$. We may assume that $t\ge 4$ since Hadwiger's conjecture holds for $t\le 3$. Let $G$ be a $K_t$-minor-free graph. It suffices to show that $\chi(G)\le C\cdot t\cdot \log\log t$. By Theorem~\ref{t:tech} applied to $G$, we find that $\chi(G)\le C_{\ref{t:tech}}\cdot t\cdot (1+f_{\ref{t:tech}}(G,t))$. 

We now upper bound $f_{\ref{t:tech}}(G,t)$ as follows. Let $a$ be an integer such that $t \ge a\ge \frac{t}{\sqrt{\log t}}$. Note that since $t\ge 4$, we have that $\frac{t}{\sqrt{\log t}}\ge 3$ and hence $a\ge 3$. Let $H$ be a $K_a$-minor-free graph with $\vert(H) \le C_{\ref{t:tech}} \cdot a\cdot \log^4 a$. Let $p:=2(a-1)$. By Theorem~\ref{t:DucMey}, $\rho(H)\le 2(a-1)=p$. Note that since $a\ge 3$, we have $p\ge a$.

Hence by Corollary~\ref{cor:smallbound}, 
$$\chi(H) \le  \left(2+\max\left\{\log \left(\frac{\vert(H)}{p}\right),~0\right\}\right)\cdot p \le (2+\log(C_{\ref{t:tech}} \cdot \log^4 a)) \cdot 2a \le 14 \cdot C_{\ref{t:tech}} \cdot a \cdot \log \log a.$$ 
Since $\log \log a$ is an increasing function in $a$ for all $a\ge 3$, we find that 
$$\frac{\chi(H)}{a} \le 14 \cdot C_{\ref{t:tech}} \cdot \log \log t$$
and hence $f_{\ref{t:tech}}(G,t) \le 14 \cdot C_{\ref{t:tech}} \cdot \log \log t$. Thus $\chi(G) \le 15 \cdot C_{\ref{t:tech}}^2 \cdot t \cdot \log \log t = C \cdot t\cdot \log \log t$ as desired.
\end{proof}

As for Corollary~\ref{cor:HadwigerSubLogClique}, we use the following result of Molloy~\cite{Molloy19} from 2019.

\begin{thm}[\cite{Molloy19}]\label{t:Molloy}
For every integer $r\ge 4$, if $G$ is a $K_r$-free graph of maximum degree at most $\Delta$, then $\chi(G) \le 200 \cdot r \cdot \Delta \cdot \frac{\log\log \Delta}{\log \Delta}$.
\end{thm}

Combined with Theorem~\ref{t:KT}, this has the following corollary.

\begin{cor}\label{cor:SmallCliqueSmallRho}
There exists $C=C_{\ref{cor:SmallCliqueSmallRho}} \ge 1$ such that the following holds: If $G$ is a $K_t$-minor-free graph, then 
$$\rho(G) \le C\cdot \omega(G) \cdot t \cdot \frac{\log \log t}{\sqrt{\log t}}.$$
\end{cor}
\begin{proof}
Let $C:=192000$. Let $H$ be a non-empty induced subgraph of $G$. Since $H$ is $K_t$-minor-free, we find by a more explicit version of Theorem~\ref{t:KT}, namely Theorem~\ref{t:density}, that $H$ has average degree $d\le 2\cdot \d(H) \le 60 \cdot t \cdot \sqrt{\log t}$. At least half of the vertices of $H$ have degree at most $2d$ and hence there exists $S\subseteq V(H)$ with $|S|\ge \vert(H)/2$ such that the maximum degree of $H[S]$ is at most $\Delta := 2\cdot (60 \cdot t \cdot \sqrt{\log t})$. Let 
$$r:=\max\{\omega(H) + 1, 4\} \le 4\cdot \omega(H) \le 4\cdot \omega(G).$$ 
Since $H$ is $K_r$-free as $r>\omega(H)$, we find by Theorem~\ref{t:Molloy} that 
$$\chi(H[S]) \le 200 \cdot r \cdot \Delta \cdot \frac{\log\log \Delta}{\log \Delta} \le 96000 \cdot \omega(G) \cdot t\sqrt{\log t} \cdot \frac{ \log\log t}{\log t}.$$ Hence there exists $I \subseteq S$ such that $I$ is an independent set in $H$ and 
$$|I| \ge \frac{|S|}{\chi(H[S])} \ge \frac{\vert(H)}{2\cdot \chi(H[S])}.$$
But then
$$\frac{\vert(H)}{|I|} \le 2\cdot \chi(H[S]) \le 192000 \cdot \omega(G) \cdot t \cdot \frac{ \log\log t}{\sqrt{\log t}}.$$
Since $\frac{\vert(H)}{\alpha(H)} \le \frac{\vert(H)}{|I|}$ and $H$ is arbitrary, it follows that $\rho(G)$ is as desired.
\end{proof}

We are now ready to derive Corollary~\ref{cor:HadwigerSubLogClique}.

\begin{proof}[Proof of Corollary~\ref{cor:HadwigerSubLogClique}.]
Let $C := 77\cdot C_{\ref{t:tech}}^2\cdot C_{\ref{cor:SmallCliqueSmallRho}}$. Let $a$ be an integer such that $t\ge a \ge \frac{t}{\sqrt{\log t}}$. Let $H$ be a $K_a$-minor-free graph with $\vert(H) \le C_{\ref{t:tech}} \cdot a \cdot \log^4 a$ and $\omega(H) \le \frac{\sqrt{\log t}}{(\log\log t)^2}$. Let $p:= 2\cdot C_{\ref{cor:SmallCliqueSmallRho}} \cdot \frac{a}{\log \log t}$. By Corollary~\ref{cor:SmallCliqueSmallRho},
$$\rho(H) \le C_{\ref{cor:SmallCliqueSmallRho}}\cdot \omega(H) \cdot a \cdot \frac{\log \log a}{\sqrt{\log a}} \le 2\cdot C_{\ref{cor:SmallCliqueSmallRho}} \cdot \frac{a}{\log \log t} = p,$$
since $\log a \ge \frac{\log t}{2}$ as $a \ge \frac{t}{\sqrt{\log t}}$.
Hence by Corollary~\ref{cor:smallbound}, we obtain that 
\begin{align*}
\chi(H) &\le \left(2+\max\left\{\log \left(\frac{\vert(H)}{p}\right),~0\right\}\right)\cdot p\\
&\le \left(2+\log\left(\frac{ C_{\ref{t:tech}} \cdot a \cdot \log^4 a}{2\cdot C_{\ref{cor:SmallCliqueSmallRho}} \cdot \frac{a}{\log \log t}} \right)\right) \cdot 2\cdot C_{\ref{cor:SmallCliqueSmallRho}} \cdot \frac{a}{\log \log t} \\
&\le (2+\log(C_{\ref{t:tech}} \cdot \log^5 t)) \cdot 2\cdot C_{\ref{cor:SmallCliqueSmallRho}} \cdot \frac{a}{\log \log t} \\
&\le 76\cdot C_{\ref{t:tech}}\cdot C_{\ref{cor:SmallCliqueSmallRho}} \cdot a.
\end{align*}
It follows that $f_{\ref{t:tech}}(G,t) \le 76\cdot C_{\ref{t:tech}}\cdot C_{\ref{cor:SmallCliqueSmallRho}}$.  By Theorem~\ref{t:tech} applied to $G$, we find that 
$$\chi(G) \le C_{\ref{t:tech}}\cdot t\cdot (1+f_{\ref{t:tech}}(G,t)) \le 77\cdot C_{\ref{t:tech}}^2\cdot C_{\ref{cor:SmallCliqueSmallRho}} \cdot t = Ct$$ 
as desired.
\end{proof}

We note that the best-known bound on the Hall ratio of $K_r$-free $K_t$-minor-free graphs is $O\left(t^{(10r-21)/(10r-20)}\right)$ by a very recent result of Buci\'c, Fox and Sudakov~\cite{BFS20} who improved and generalized the recent work of Dvo\v{r}\'ak and Yepremyan~\cite{DY19} who proved a bound of $O(t^{26/27})$ for the Hall ratio of triangle-free $K_t$-minor-free graphs.

\section{Small Dense Subgraphs}\label{s:smallConn}

In this section, we prove Theorem~\ref{t:SmallConn2} which proves the existence in a $K_t$-minor-free graph $G$ of a small dense subgraph with only a constant loss in density. 

\subsection{Proof Overview}

Recall that we use a new approach where we attempt to find a $K_t$ minor in only one stage and when that fails, we find a small dense subgraph.  More specifically, we successively contract components with roughly $\log^2 t$ vertices as long as not too many are edges are lost (say $d/10$ per contraction) to obtain a new graph $G'$. The key idea in the proof is to limit the next contraction to a special non-empty set $S'$ of vertices such that each vertex of $S'$ has
\begin{itemize}
\item[(1)] not too large degree in $G'$ (roughly $d \cdot \log t)$, 
\item[(2)] not too many neighbors (say $d/20$) among the previously contracted components, 
\item[(3)] many neighbors in $S'$ (say $2d/5$), and 
\item[(4)] a linear proportion (say $1/3$) of its degree in $G'$ in $S'$. 
\end{itemize}
Thus when the contraction process fails (which it inevitably does as $G$ is $K_t$-minor-free) for some new component $x$, each neighbor of $x$ in $S'$ will by (2) and (3) have many common neighbors with $x$ among the uncontracted vertices; yet by (4), it can be shown that $x$ has a linear proportion of its degree in $G'$ in $S'$. By (1), the number of neighbors of $x$ is $O(d\cdot \log^3 t)$ and hence the uncontracted neighbors of $x$ form the desired small dense subgraph. 

The existence of a set $S'$ all of whose vertices have properties (3) and (4) is the consequence of a remarkable little lemma below (Lemma~\ref{lem:SpecialSubset}). To guarantee properties (1) and (2), we need to handle bipartite subgraphs; instead of handling bipartite subgraphs as a separate case as in the proof of Theorem~\ref{t:newforced}, we exploit the following unbalanced bipartite density theorem of $K_t$-minor-free graphs by Norin and the second author~\cite{NorPos20}.

\begin{thm}\label{t:logbip}
	There exists $C=C_{\ref{t:logbip}}\ge 1$ such that for every integer $t \geq 3$ and every bipartite graph $G$ with bipartition $(S,T)$ and no $K_t$ minor, we have
	\begin{equation*}
	\e(G) \le C t\sqrt{\log t} \sqrt{|S||T|}  + (t-2)\vert(G).
	\end{equation*}	 
\end{thm}	 

\subsection{Proof of Theorem~\ref{t:SmallConn2}}

Before we prove Theorem~\ref{t:SmallConn2}, here is that remarkable little lemma that holds for all graphs.

\begin{lem}\label{lem:SpecialSubset}
Let $r>2$ and $\delta >0$. If $G$ is a graph and $S$ is a subset of $V(G)$ such that
$$(r-2)\cdot \e(G[S]) > (r-1)\cdot \delta |S| + \e(G(S,V(G)\setminus S)),$$
then there exists a nonempty subset $S'$ of $S$ such that both of the following hold:
\begin{itemize}
\item[(i)] the minimum degree of $G[S']$ is at least $\delta$, and
\item[(ii)] $|N_G(v)\cap S'| \ge \frac{\deg_G(v)}{r}$ for every $v\in S'$.
\end{itemize}  
\end{lem}
\begin{proof}
Let $S'$ be a subset of $S$ satisfying
$$(r-2)\cdot \e(G[S']) > (r-1)\cdot \delta |S'| + \e(G(S',V(G)\setminus S')),$$
and subject to that, $|S'|$ is minimized. Note then that $\e(G[S'])>0$ and hence $S'$ is non-empty. If $S'$ satisfies both conditions (i) and (ii), then $S'$ is as desired.

So we may assume that $S'$ violates at least one of conditions (i) or (ii). First suppose $S'$ violates condition (i). That is, there exists a vertex $v\in S'$ such that $|N_G(v)\cap S'| < \delta$. Let $S'' := S'\setminus \{v\}$. Now 
$$\e(G[S'']) \ge \e(G[S']) - \delta$$ 
and 
$$\e(G(S'',V(G)\setminus S'')) \le \e(G(S',V(G)\setminus S'))+\delta.$$ 
Meanwhile, $|S''|=|S'|-1$. Hence
\begin{align*}
(r-2)\cdot \e(G[S'']) &\ge (r-2)\cdot (\e(G[S']) - \delta) \\
&> (r-1)\cdot \delta |S'| + \e(G(S',V(G)\setminus S')) - (r-2) \delta \\
&\ge (r-1)\cdot \delta (|S'|-1) + \e(G(S',V(G)\setminus S'))+\delta \\
&\ge (r-1)\cdot \delta |S''| + \e(G(S'',V(G)\setminus S'')),
\end{align*}
and hence $S''$ contradicts the choice of $S'$.

So we may assume that $S'$ violates condition (ii). That is, there exists a vertex $v\in S'$ such that $|N_G(v)\cap S'| < \frac{\deg_G(v)}{r}$. Let $S'' := S'\setminus \{v\}$. Now 
$$\e(G[S'']) \ge \e(G[S']) - \frac{\deg_G(v)}{r}$$ and 
$$\e(G(S'',V(G)\setminus S'')) \le \e(G(S',V(G)\setminus S'))-\frac{r-2}{r}\cdot \deg_G(v).$$ 
Meanwhile, $|S''|=|S'|-1$. Hence
\begin{align*}
(r-2)\cdot \e(G[S'']) &\ge (r-2)\cdot \left(\e(G[S']) - \frac{\deg_G(v)}{r}\right) \\
&> (r-1)\cdot \delta |S'| + \e(G(S',V(G)\setminus S'))-\frac{r-2}{r}\cdot \deg_G(v) \\
&\ge  (r-1)\cdot \delta |S''| + \e(G(S'',V(G)\setminus S'')),
\end{align*}
and hence $S''$ contradicts the choice of $S'$.
\end{proof}

We also need two classical results as follows. The first is an explicit form of Theorem~\ref{t:KT}, which is not hard to derive from the results of Kostochka in~\cite{Kostochka82}; namely, he showed that for $t\ge 4$, the maximum density $d$ of a $K_t$-minor-free graph satisfies $t\ge \frac{0.064 \cdot d}{\sqrt{\log d}}$.

\begin{thm}\label{t:density}
 	Let $t \geq 2$ be an integer. Then every graph $G$  with $\d(G) \geq 30 \cdot t\cdot \sqrt{\log t}$ has a $K_t$ minor.
\end{thm}

The second is a result of Mader~\cite{Mader72} that shows that high density graphs have a highly-connected subgraph.

\begin{lem}[\cite{Mader72}]\label{l:connect}
Every graph $G$ contains a subgraph $G'$ such that $\kappa(G') \geq \d(G)/2$.
\end{lem}

We are now ready to prove Theorem~\ref{t:SmallConn2} which we restate for convenience.

\SmallConnNew*

\begin{proof}[Proof of Theorem~\ref{t:SmallConn2}]
Let $C := 480 \cdot C_{\ref{t:logbip}}.$ Let $d :=Ck$. We note the result is vacuously true if $t\le 2$ since there does not exist a $K_2$-minor-free graph $G$ with $\d(G)\ge 1$. So we assume that $t\ge 3$. By deleting edges as necessary, we may assume without loss of generality that $\d(G)=d$. Note by Theorem~\ref{t:density} since $G$ is $K_t$-minor-free, we have that $d < 30 \cdot t \cdot \sqrt{\log t}$. Since $C \ge 30$, it follows that $k\le t\cdot \sqrt{\log t}$.

Let $H_1, \ldots, H_m$ be non-empty pairwise vertex-disjoint connected subgraphs of $G$ such that $\vert(H_i)= \left\lceil 10 \cdot \log^2 t \cdot \left(\frac{t}{k}\right)^2 \right\rceil$ for every $i\in [m]$ and the graph $G'$ obtained from $G$ by contracting each $H_i$ to a new vertex $x_i$ for each $i\in [m]$ satisfies
$$\e(G) - \e(G') \le \frac{d}{10} \cdot \big(\vert(G)-\vert(G')\big),$$
and subject to that $m$ is maximized. Note that such a collection exists since the empty collection satisfies the hypotheses. 

Note that $G'$ is $K_t$-minor-free since $G$ is. Let 
\begin{align*}
X &:=\{x_i: i\in [m]\},\\
Y &:= \{v\in V(G')\setminus X: |N_{G'}(v) \cap X| \ge d/20 \},\text{ and }\\
Z &:= \{v\in V(G')\setminus (X\cup Y): \deg_{G'}(v) \ge 20 \cdot d \cdot \log t\}.
\end{align*} 

Since $\vert(H_i)\ge 10 \cdot \log^2 t \cdot \left(\frac{t}{k}\right)^2$ for every $i\in [m]$ and the $H_i$ are pairwise vertex-disjoint, we have that $|X| \le \frac{\vert(G)}{10 \cdot \log^2 t} \cdot \left(\frac{k}{t}\right)^2$.

\newtheorem{innercustomthm}{Claim}
\newenvironment{customthm}[1]
  {\renewcommand\theinnercustomthm{#1}\innercustomthm}
  {\endinnercustomthm}
	
\begin{customthm}{\ref{t:SmallConn2}.1}
$|Y|\le (\log t)\cdot |X| \cdot \left(\frac{t}{k}\right)^2$.
\end{customthm}
\begin{proof}
Since $t\ge 3$, if $|Y|\le |X|$, then $|Y|\le (\log t)\cdot |X| \cdot \left(\frac{t}{k}\right)^2$  as desired since $k\le t\cdot \sqrt{\log t}$. So we may assume that $|X| < |Y|$. Since $G'$ is $K_t$-minor-free, we have by Theorem~\ref{t:logbip} applied to $G'(X,Y)$ that
$$\e(G'(X,Y)) \le C_{\ref{t:logbip}} \cdot t\sqrt{\log t} \sqrt{|X||Y|}  + (t-2)(|X|+|Y|).$$
Since $d \ge 60\cdot C_{\ref{t:logbip}} \cdot k$, it follows from the definition of $Y$ that 
$$\e(G'(X,Y)) \ge \frac{d}{20}\cdot |Y| \ge 3 \cdot C_{\ref{t:logbip}} \cdot k \cdot |Y|.$$ 
Using $|X|\le |Y|$, $t\le k$ and $C_{\ref{t:logbip}}\ge 1$, we find that $(t-2)(|X|+|Y|) \le 2\cdot C_{\ref{t:logbip}} \cdot k \cdot |Y|$. Combining these bounds and rearranging, we then find that
$$C_{\ref{t:logbip}} \cdot k \cdot |Y| \le C_{\ref{t:logbip}} \cdot t\sqrt{\log t}\sqrt{|X||Y|},$$
which squaring both sides and canceling implies that $|Y| \le (\log t) \cdot |X|\cdot \left(\frac{t}{k}\right)^2$ as desired.
\end{proof}

Note that by assumption $\e(G') \ge \e(G) - \frac{d}{10} \cdot \vert(G)$. Since $\e(G) = d \cdot \vert(G)$, we have that $\e(G') \ge \frac{9}{10} \cdot \e(G)$. On the other hand,  since $\e(G') \le \e(G)=d\cdot \vert(G)$, we have by the definition of $Z$ that $|Z| \le 2\cdot \frac{\vert(G)}{20 \log t}$. 

Let $T :=X\cup Y\cup Z$ and let $S :=V(G')\setminus T$. Thus 
\begin{align*}
|T| &= |X|+|Y|+|Z| \le |X|\cdot \left(1+(\log t)\cdot \left(\frac{t}{k}\right)^2\right) + |Z| \\ 
&\le 2\log t \cdot \left(\frac{t}{k}\right)^2\cdot |X| + \frac{\vert(G)}{10 \log t} \le \frac{3}{10} \cdot \frac{\vert(G)}{\log t}.
\end{align*}
Since $G'$ is $K_t$-minor-free, we have by Theorem~\ref{t:logbip} applied to $G'(S, T)$ that
$$\e(G'(S,T)) \le  C_{\ref{t:logbip}} \cdot t\sqrt{\log t} \sqrt{|S||T|}  + (t-2)\cdot \vert(G'),$$
and hence 
$$\e(G'(S,T)) \le C_{\ref{t:logbip}} \cdot t\sqrt{\log t} \cdot \sqrt{|S|\cdot \frac{3\cdot \vert(G)}{10\cdot \log t}}  + (t-2)\cdot \vert(G) \le C_{\ref{t:logbip}} \cdot (2t-2)\cdot \vert(G),$$ 
since $|S|\le \vert(G)$.

Since $G$ is $K_t$-minor-free, we have by Theorem~\ref{t:density} that 
$$\e(G'[T]) \le 30 \cdot t\sqrt{\log t} \cdot |T| \le 9\cdot t\cdot \vert(G).$$ 
Combining, we find that 
$$\e(G'[T])+\e(G'(S,T)) < 11 \cdot C_{\ref{t:logbip}} \cdot t \cdot \vert(G) \le \frac{1}{20} \cdot \e(G),$$ 
since $\e(G) = d\cdot \vert(G) \ge 220 \cdot C_{\ref{t:logbip}} \cdot t \cdot \vert(G)$.  Thus
$$\e(G'[S]) = \e(G')-\e(G'[T])-\e(G'(S,T)) > \frac{9}{10}\cdot \e(G) - \frac{1}{20} \cdot \e(G) = \frac{17}{20}\cdot \e(G) = \frac{17}{20} \cdot d \cdot \vert(G).$$ 
Let $r=3$ and $\delta = \frac{2}{5} \cdot d$. From the above calculations, we have that
$$\e(G'[S]) > 2\delta \cdot |S| + \e(G'(S,V(G')\setminus S)).$$
Thus by Lemma~\ref{lem:SpecialSubset}, there exists a nonempty subset $S'$ of $S$ such that both of Lemma~\ref{lem:SpecialSubset}(i) and (ii) hold; that is, the minimum degree of $G[S']$ is at least $\delta$ and $|N_{G'}(v)\cap S'| \ge \frac{1}{3} \cdot \deg_{G'}(v)$ for every $v\in S'$. 

Now let $H_{m+1}$ be a non-empty connected subgraph of $S'$ such that $\vert(H_{m+1})\le \left\lceil 10 \cdot \log^2 t \left(\frac{t}{k}\right)^2 \right\rceil$ and the graph $G''$ obtained from $G'$ by contracting $H_{m+1}$ to a new vertex $x_{m+1}$ satisfies 
$$\e(G') - \e(G'') \le \frac{d}{10} \cdot (\vert(H_{m+1})-1),$$
and subject to that $\vert(H_{m+1})$ is maximized. Note that $H_{m+1}$ exists as any vertex in $S'$ satisfies the conditions.

By the maximality of $m$, we find that $\vert(H_{m+1}) \ne \left\lceil 10 \cdot \log^2 t \left(\frac{t}{k}\right)^2 \right\rceil$ and hence $\vert(H_{m+1}) \le \left\lceil 10 \cdot \log^2 t \left(\frac{t}{k}\right)^2 \right\rceil - 1$. Let $R := N_{G''}(x_{m+1})\setminus X$ and $R':=N_{G''}(x_{m+1})\cap S'$. Note that
$$|R'| \ge \left(\sum_{v\in V(H_{m+1})} |N_{G'}(v)\cap S'|\right) - 2 \cdot (\e(G')-\e(G'')).$$
By condition Lemma~\ref{lem:SpecialSubset}(i), we find that $|N_{G'}(v)\cap S'|\ge \delta = \frac{2}{5} \cdot d$ for each $v\in V(H_{m+1})$. Since $\e(G')-\e(G'') \le \frac{d}{10} \cdot \vert(H_{m+1})$, we then find that
$$|R'| \ge \frac{1}{2} \left(\sum_{v\in V(H_{m+1})} |N_{G'}(v)\cap S'|\right).$$
Thus by condition Lemma~\ref{lem:SpecialSubset}(ii) since $r=3$, we have that 
$$|R'| \ge \frac{1}{6}\left(\sum_{v\in V(H_{m+1})} \deg_{G'}(v) \right) \ge \frac{|R|}{6}.$$  

Let $v \in R'$. By definition of $Y$, we have that $|N_{G''}(v)\cap X| \le \frac{d}{20}$. Let $H_{m+1}' := G[V(H_{m+1})\cup \{v\}]$. Note that $\vert(H_{m+1}') = \vert(H_{m+1}) + 1 \le \left\lceil 10 \cdot \log^2 t \left(\frac{t}{k}\right)^2 \right\rceil$. Note by definition of $R'$, we have that $H_{m+1}'$ is connected. Thus by the maximality of $H_{m+1}$, we find that 
$$\e(G')-\e(G'') > \frac{d}{10}.$$
Since $\frac{d}{10}$ is integral by definition of $d$, we find that $\e(G')-\e(G'')\ge \frac{d}{10} + 1$. Hence we find that
$$|N_{G''}(v)\cap N_{G''}(x_{m+1})| \ge \e(G')-\e(G'') - 1 \ge \frac{d}{10},$$
where the $-1$ accounts for the edge $vx_{m+1}$. 
Hence for every $v\in R'$, we have that
$$|N_{G''}(v)\cap R| \ge \frac{d}{10} - \frac{d}{20} = \frac{d}{20}.$$ 
Since every vertex in $R$ in $G''$ corresponds to a vertex in $G$, we find that $|N_{G}(v)\cap R| = |N_{G''}(v)\cap R|$ for every $v\in R'$ and that $G[R]=G''[R]$. But then 
$$\e(G[R]) \ge \frac{1}{2}\cdot \frac{d}{20} \cdot |R'| \ge \frac{d}{240} \cdot |R|.$$
Thus $\d(G[R]) \ge \frac{d}{240} \ge 2k$ since $C\ge 480$. Since every vertex in $H_{m+1}$ is not in $Z$, we find that 
$$|R| \le 20\cdot d\cdot (\log t) \cdot \vert(H_{m+1}) \le 200 \cdot Ck \cdot \log^3 t \cdot \left(\frac{t}{k}\right)^2 \le C^2 \cdot t\cdot \log^3 t,$$
since $d=Ck$, $k\ge t$ and $C\ge 200$. By Lemma~\ref{l:connect}, $G[R]$ has a $k$-connected subgraph $H$. Since $H$ is a subgraph of $G[R]$, we have that $\vert(H)\le |R| \le C^2 \cdot t\cdot \log^3 t$ and hence $H$ is a subgraph as desired.  
\end{proof}

\section{Connectivity Toolkit}\label{s:prelim}

To prove Lemma~\ref{lem:inseparable2} and Theorem~\ref{t:tech}, we need a suite of connectivity results as follows.

First we need the following recent yet crucial result of Gir\~{a}o and Narayanan~\cite{GirNar20} which shows that we can exchange a small amount of chromatic number to obtain a subgraph with high-connectivity. In particular the gained connectivity is linear in the forfeited chromatic number. We note that while Gir\~{a}o and Narayanan did not explicitly calculate the constant in the theorem below, the theorem is nevertheless implicit in their work.

\begin{thm}[\cite{GirNar20}]\label{t:LargeChi}
For every positive integer $k$, if $G$ is a graph with $\chi(G)\ge 7k$, then $G$ contains a $k$-connected subgraph $H$ with $\chi(H)\ge \chi(G)-6k$.
\end{thm}

Indeed, Theorem~\ref{t:LargeChi} is very useful and we will use it whenever we need to restore connectivity to continue building a complete minor. We remark that Nguyen~\cite{Ngu22} later improved the constants in Theorem~\ref{t:LargeChi}, though for our purposes any constant will do. 

Another useful lemma is the following ``redundancy" version of Menger's Theorem. First recall the following definitions. Let $G$ be a graph and $A,B\subseteq V(G)$. A path $P=v_1\ldots v_k$ in $G$ is an \emph{$A-B$ path} if $V(G)\cap A = \{v_1\}$ and $V(G)\cap B = \{v_k\}$. A subset $X\subseteq V(G)$ \emph{separates} $A$ and $B$ if every $A-B$ path contains a vertex of $X$. Finally we recall Menger's Theorem from 1927 before stating our variant.

\begin{thm}[Menger's Theorem~\cite{Men27}]
Let $G$ be a graph and $A,B\subseteq V(G)$. Then the minimum number of vertices separating $A$ from $B$ is equal to the maximum number of vertex-disjoint $A-B$ paths.
\end{thm} 

\begin{lem}\label{lem:MengerVariant}
Let $G$ be a graph and let $A_1,A_2,B$ be disjoint non-empty subsets of $V(G)$. If for each $i\in\{1,2\}$, there exists a
set $\mc{P}_i$ of $A_i-B$ paths in $G\setminus A_{3-i}$ with $|\mc{P}_i|=2|A_i|$ which are pairwise disjoint in $G\setminus A_i$ such that every vertex of $A_i$ is in exactly two paths in $\mc{P}_i$, then there exist $|A_1|+|A_2|$ vertex-disjoint $(A_1\cup A_2) - B$ paths in $G$.
\end{lem}
\begin{proof}
Suppose not. Then by Menger's Theorem there exists a set $X$ separating $A_1\cup A_2$ and $B$ in $G$ such that $|X| < |A_1|+|A_2|$. Let $A_1' :=A_1\setminus X$, $A_2' :=A_2\setminus X$ and let $X' := X\setminus (A_1\cup A_2)$. It follows that 
$$|X'|  = |X| - |A_1\cap X| - |A_2\cap X| < |A_1|+|A_2| - |A_1\cap X| - |A_2\cap X| = |A_1'|+|A_2'|.$$ 
Assume without loss of generality that $|A_1'| \ge |A_2'|$. Then $|X'| < 2|A_1'|$. But now $X'$ separates $A_1'$ from $B$ in $G' := G-(A_2 \cup (A_1\cap X))$. However, the subset $\mc{P}'_1$ of $\mc{P}_1$ consisting of paths with ends in $A_1'$ has size at least $2|A_1'|$. As the paths in $\mc{P}'_1$ are vertex-disjoint, it follows that there exists a path in $\mc{P}'_1$ disjoint from $X'$, a contradiction since $X'$ intersects every $A_1'-B$ path in $G'$ by definition.
\end{proof}

Lemma~\ref{lem:MengerVariant}, while quite a natural variation of Menger's Theorem, has not to our knowledge previously appeared in the literature. We note that Lemma~\ref{lem:MengerVariant} easily generalizes from 2 sets to $k$ sets, though we omit the details. We think this ``redundancy version" is of interest in its own right and could have a number of applications. For our purposes, when building a complete minor, it allows us to connect our previously built parts of the minor as well as our future workspace to the present workspace simply by adding some redundancy to the paths connecting them.

Next we need a suite of results on connectivity. First, we recall some definitions.

\begin{definition}
Let $\ell$ be a positive integer. Let $G$ be a graph and let $\mc{S} = \{(s_i,t_i)\}_{i \in [\ell]}$ be a collection of pairs of vertices of a graph $G$ such that for distinct $i,j\in [\ell]$, we have that $s_i\ne s_j$, $t_i\ne t_j$ , and further if $s_i=t_j$ for some $i,j \in [\ell]$, then $i=j$. An \emph{$\mc{S}$-linkage $\mc{P}$} is a collection of vertex-disjoint paths $\{P_1,\ldots,P_{\ell}\}$ in $G$ such that $P_i$ has ends $s_i$ and $t_i$ for every $i \in [\ell]$. We let $V(\mc{P})$ denote the set $\bigcup_{i=1}^{\ell} V(P_i)$.
\end{definition}

\begin{definition}
A graph G is said to be \emph{$\ell$-linked} if $1 \le \ell \le \vert(G)$ and for any collection $\mc{S} = \{(s_i,t_i)\}_{i \in [\ell]}$ of pairs of vertices of a graph $G$ such that for distinct $i,j\in [\ell]$, we have that $s_i\ne s_j$ and $t_i\ne t_j$, and further if $s_i=t_j$ for some $i,j \in [\ell]$, then $i=j$, there exists an $\mc{S}$-linkage.
\end{definition}

Note that an $\ell$-linked graph is $\ell$-connected. We will need the following result of Bollob\'{a}s and Thomason~\cite{BolTho96} which shows the converse is true up to a constant factor.

\begin{thm}[\cite{BolTho96}]\label{t:linked}
There exists an integer $C=C_{\ref{t:linked}}\ge 1$ such that the following holds: Let $\ell$ be a positive integer. If $G$ is a graph with $\kappa(G) \geq C\ell$, then $G$ is $\ell$-linked.
\end{thm}

For reference, the value of $C_{\ref{t:linked}}$ is not explicitly given in~\cite{BolTho96}, but it is not hard to see from their work that $C_{\ref{t:linked}}=22$ suffices. We note that Thomas and Wollan~\cite{ThoWol05} improve the bounds from~\cite{BolTho96}, and indeed the results of~\cite{ThoWol05} directly imply that $C_{\ref{t:linked}}=10$ satisfies \cref{t:linked}.

We also desire a version where instead of pairs of vertices being connected, we desire sets of vertices to be connected as follows.

\begin{definition}
A graph G is said to be \emph{$(k,\ell)$-knit} if $1 \le \ell \le k \le \vert(G)$ and, whenever $S$ is a set of $k$ vertices of $G$ and $S_1,\ldots,S_t$ is a partition of $S$ into $t \ge \ell$ non-empty parts, then $G$ contains vertex-disjoint connected subgraphs $D_1,\ldots,D_t$ such that $S_i \subseteq V(D_i)$ for each $i\in [t]$. Clearly, a $(2k,k)$-knit graph is $k$-linked.
\end{definition}

The following is an easy corollary of Theorem~\ref{t:linked}.

\begin{thm}\label{t:knitted}
 There exists an integer $C=C_{\ref{t:knitted}}\ge 1$ such that the following holds: Let $k\ge \ell$ be positive integers. If $G$ is a graph with $\kappa(G) \geq Ck$, then $G$ is $(k,\ell)$-knit.
\end{thm}

Next we require the following theorem of Wollan~\cite{W08} on rooted minors.

\begin{thm}[\cite{W08}]\label{t:rootedMinors}
Let $H$ be a fixed graph and $c\ge 1$ be a real number such that every graph $G$ with $\d(G)\ge c$ contains $H$ as a minor. If $G$ is a $\vert(H)$-connected graph with $\d(G)\ge 9c+26833 \cdot \vert(H)$, then for all sets $X\subseteq V(G)$ with $|X|=\vert(H)$ and all bijective maps $\pi: X\rightarrow V(H)$, then there exists an $H$ model in $G$ $\pi$-rooted at $X$.
\end{thm}

We also need the following easy corollary of Theorem~\ref{t:density}. We let $K_s+ \ell K_2$ denote the disjoint union of $K_s$ and a matching of size $\ell$.

\begin{cor}
Let $s \ge 2$  and $\ell \ge 0$ be integers. Every graph $G$ with $\d(G) \ge 30 \cdot s\cdot \sqrt{\log s} + 2 \ell$ has a $K_s + \ell K_2$ minor. 
\end{cor}
\begin{proof}
We proceed by induction on $\ell$. If $\ell = 0$, then the result follows from Theorem~\ref{t:density}. So we may assume that $\ell \ge 1$. Since $\d(G) > 0$, there exists an edge $e=uv$ of $G$. Let $G':=G\setminus \{u,v\}$. Note that $\d(G') \ge d(G) - 2 \ge 30 \cdot s\cdot \sqrt{\log s} + 2(\ell-1)$. Hence by induction, $G'$ has a $K_s + (\ell-1) K_2$ minor $M$. But then $M\cup \{uv\}$ is a $K_s + \ell K_2$ minor as desired.
\end{proof}

Now applying Theorem~\ref{t:rootedMinors} with $H= K_s + \ell K_2$ and $c= 30 \cdot s\cdot \sqrt{\log s} + 2 \ell$ yields the following lemma.

\begin{lem}\label{l:rooted} There exists $C=C_{\ref{l:rooted}} >0$ such that the following holds:
	Let $G$ be a graph,  let $\ell \geq s \geq 2$ be positive integers.  Let $s_1,\ldots, s_{\ell},$ $t_1,\ldots,t_{\ell},$  $r_1,\ldots, r_s \in V(G)$ be distinct, except possibly $s_i=t_i$ for some number of $i \in [\ell]$.
	If $$\kappa(G) \geq C \cdot\max\{\ell, s\sqrt{\log s}\},$$	
	then there exists a $K_{s}$ model $\mc{M}$ in $G$ rooted at $\{r_1,\ldots,r_s\}$ and an $\{(s_i,t_i)\}_{i \in [\ell]}$-linkage $\mc{P}$ in $G$ such that $\mc{M}$ and $\mc{P}$ are vertex-disjoint. 
\end{lem}

For convenience, we desire a slight strengthening. To simplify its statement, we introduce a new term \emph{woven} as follows. This concept is used in the both of the proofs of Lemma~\ref{lem:inseparable2} and Theorem~\ref{t:tech}.

\begin{definition}
Let $a,b$ be nonnegative integers. We say a graph $G$ is \emph{$(a,b)$-woven} if the following holds:
for every three sets of vertices $R=\{r_1,\ldots, r_a\}$, $S=\{s_1,\ldots, s_b\},$ $T=\{t_1,\ldots,t_b\}$ in $V(G)$ such that if $s_i=t_j$ for some $i,j \in [b]$, then $i=j$, there exists a $K_{a}$ model $\mc{M}$ in $G$ rooted at $R$ and an $\{(s_i,t_i)\}_{i \in [b]}$-linkage $\mc{P}$ in $G$ such that $V(\mc{M})\cap V(\mc{P}) = R\cap (S\cup T)$. 
\end{definition}   

Here is the strengthening.

\begin{lem}\label{l:rooted3} 
There exists an integer $C = C_{\ref{l:rooted3}} \ge 1$ such that the following holds: Let
$a,b \ge 2$ be integers. If $G$ is a graph with
$$\kappa(G) \geq C \cdot\max\{a\sqrt{\log a}, b\},$$
then $G$ is $(a,b)$-woven.
\end{lem}
\begin{proof}
Let $C_{\ref{l:rooted3}} := \max \{C_{\ref{l:rooted}}+2, 4\}$. Since $C_{\ref{l:rooted3}}\ge 4$, we have that $G$ is $(2a+2b)$-connected and hence $\delta(G)\ge 2a+2b$. 

Let $R=\{r_1,\ldots, r_a\}$, $S=\{s_1,\ldots, s_b\},$ $T=\{t_1,\ldots,t_b\}$ be subsets of $V(G)$ such that if $s_i=t_j$ for some $i,j \in [b]$, then $i=j$.

Since $\delta(G)\ge 2a+2b$, there exists a set $R' = \{r'_i: i\in [a]\}$ disjoint from $R\cup S\cup T$ such that $r'_i$ is a neighbor of $r_i$ for every $i\in [a]$. Let $G':=G\setminus (R\setminus (S\cup T))$. Since $G'$ is $C_{\ref{l:rooted}}$-connected as $C_{\ref{l:rooted3}}\ge C_{\ref{l:rooted}}+2$, we have by Lemma~\ref{l:rooted} that there exists a $K_{a}$ model $\mc{M}'$ in $G'$ rooted at $R'$ and an $\{(s_i,t_i)\}_{i \in [b]}$-linkage $\mc{P}$ in $G'$ such that $\mc{M'}$ and $\mc{P}$ are disjoint. 

For each $i\in [a]$, let $M'_i$ be the subgraph in $\mc{M'}$ containing $r'_i$ and then let $M_i := M'_i + r_ir'_i$. Now let $\mc{M} := \{M_i: i\in [a]\}$. Note that $\mc{M}$ is a $K_a$ model in $G$ rooted at $R$ and $V(\mc{M})\cap V(\mc{P}) = R\cap (S\cup T)$. Thus $G$ is $(a,b)$-woven as desired.
\end{proof}

We note that Lemma~\ref{l:rooted3} implies Theorem~\ref{t:linked} by setting $a=2$; similarly setting $b=2$, Lemma~\ref{l:rooted3} yields a rooted version (assuming connectivity instead of average degree) of Theorem~\ref{t:KT} up to the constant factor. 

Next we need the following lemma of Kawarabayashi~\cite{Kaw07}, whose proof was inspired by the proof idea of Robertson and Seymour~\cite{GM13}.

\begin{lem}[\cite{Kaw07}]\label{lem:MinorToWoven}
Let $a\ge 1$ be an integer. If $G$ is a graph with $\kappa(G)\ge a$ and $G$ has a $K_{2a}$ minor, then $G$ is $(a,0)$-woven. 
\end{lem}

Finally we need the following useful lemma about woven subgraphs which shows that if a subgraph $H$ of a graph $G$ is $(a,b)$-woven then one can ``weave" inside $H$ a rooted $K_a$ minor while preserving a given $b$-linkage in $G$. This lemma is quite useful for building large complete minors from small woven subgraphs.

\begin{lem}\label{l:woven}
Let $a,b$ be positive integers. Let $G$ be a graph. Let $S=\{s_1,\ldots, s_b\}$, $T=\{t_1,\ldots, t_b\}$ be disjoint sets of vertices in $G$. Let $\mc{P}$ be an $\{(s_i,t_i)\}_{i \in [b]}$-linkage in $G$. If $H$ is a subgraph of $G$ that is $(a,b)$-woven and $R=\{r_1,\ldots, r_a\}\subseteq V(H)$, then there exists an $\{(s_i,t_i)\}_{i \in [b]}$-linkage $\mc{P'}$ in $G$ and a $K_a$ model $\mc{M}$ in $G$ rooted at $R$ such that $V(\mc{P'})\cap V(\mc{M}) \subseteq R\cap V(\mc{P})$ and $V(\mc{P'})\subseteq V(H)\cup V(\mc{P})$.
\end{lem}
\begin{proof}
For $i\in [b]$, let $P_i$ denote the path in $\mathcal{P}$ containing $\{s_i,t_i\}$. Let $I:=\{i\in[b]: V(P_i)\cap V(H)\ne \emptyset \}$. For $i\in[I]$, let $s'_i$ be the vertex in $V(P_i)\cap V(H)$ closest to $s_i$ in $P_i$, and similarly let $t'_i$ be the vertex in $V(P_i)\cap V(H)$ closest to $t_i$ in $P_i$. 

Since $H$ is $(a,b)$-woven, there exists a $K_{a}$ model $\mc{M}$ in $H$ rooted at $R$ and an $\{(s'_{i},t'_{i})\}_{i \in I}$-linkage $\mc{P}_1$ in $H$ such that $V(\mc{M})\cap V(\mc{P}_1) = R\cap (\{s'_i:i\in I\}\cup \{t'_i:i\in I\}) \subseteq R\cap V(\mc{P})$. 

Now for each $i\in I$, let $P'_i$ be the path obtained from concatenating the subpath of $P_i$ from $s_i$ to $s'_i$, the path in $\mc{P}_1$ connecting $s'_i$ to $t'_i$, and the subpath of $P_i$ from $t'_i$ to $t_i$. Note that by construction $V(\mc{P'}) \subseteq V(H)\cup V(\mc{P})$. Hence $\mc{P}:=(P_i: i\in [b])$ and $\mc{M}$ are as desired.
\end{proof}

\section{Finding a Minor in Chromatic-Inseparable Graphs}\label{s:inseparable}

In this section, we prove Lemma~\ref{lem:inseparable2} by proving a stronger inductive form, Lemma~\ref{lem:inseparable3}. First we need some auxiliary lemmas.

\subsection{Auxiliary Lemmas}

First we need the following corollary of Theorem~\ref{t:SmallConn2} which allows us to extract many small highly-connected subgraphs from a $K_t$-minor-free graph provided the chromatic number is large enough.

\begin{cor}\label{c:SmallConn}
Let $t\ge 3$ be an integer. Let $k\ge t$ be an integer and let $r=\lceil \sqrt{\log t} \rceil$. Let $G$ be a graph and let
$$g_{\ref{c:SmallConn}}(G,t) := \max_{H\subseteq G} \left\{ \frac{\chi(H)}{t}: \vert(H)\le C_{\ref{t:SmallConn2}}^2\cdot  t\cdot  \log^4 t,~H \text{ is $K_t$-minor-free }\right\}.$$ 
If $\chi(G) \ge 4\cdot C_{\ref{t:SmallConn2}} \cdot k\cdot (1+g_{\ref{c:SmallConn}}(G,t))$ and $G$ contains no $K_t$ minor, then $G$ contains $r$ vertex-disjoint $k$-connected subgraphs $H_1,\ldots, H_{r}$ with $\vert(H_i) \le C_{\ref{t:SmallConn2}}^2\cdot t \cdot \log^3 t$ for every $i\in [r]$.
\end{cor}
\begin{proof}
Suppose not. Let $H_1, \ldots, H_s$ be a maximal collection of vertex-disjoint $k$-connected subgraphs such that $\vert(H_i) \le C_{\ref{t:SmallConn2}}^2\cdot t \cdot \log^3 t$ for every $i\in [s]$. Since the lemma does not hold, $s < r$. Let $G_1 := G[\bigcup_{i=1}^{s} V(H_i)]$ and $G_2 := G\setminus V(G_1)$.

By the maximality of $H_1,\ldots, H_s$ and Theorem~\ref{t:SmallConn}, it follows that $\d(G_2') < C_{\ref{t:SmallConn2}}\cdot k$ for every subgraph $G_2'$ of $G_2$. Thus by greedy coloring, we have that 
$$\chi(G_2) \le 2\cdot \max_{G_2'\subseteq G_2} \d(G_2') + 1 \le 3 \cdot C_{\ref{t:SmallConn2}}\cdot k.$$ 
On the other hand, 
$$\vert(G_1) = \sum_{i=1}^{s} \vert(H_i) \le s \cdot C_{\ref{t:SmallConn2}}^2\cdot t \cdot \log^3 t \le C_{\ref{t:SmallConn2}}^2\cdot t \cdot \log^4 t.$$
Hence by definition, 
$$\chi(G_1) \le t\cdot g(G,t).$$
Thus
$$\chi(G) \le \chi(G_1) + \chi(G_2) \le 3 \cdot C_{\ref{t:SmallConn2}}\cdot k + t\cdot g(G,t),$$
a contradiction.
\end{proof}

Next we need a lemma that shows a connected graph with a prescribed set of terminals has an induced connected subgraph containing the terminals which has small chromatic number except for a set that has linear size in the terminals (namely, it has chromatic number two since it is a disjoint union of paths). This is useful when we need to reoptimize the minor after each stage of the sequential construction to ensure we do not lose too much chromatic number overall.

\begin{lem}\label{lem:Steiner}
If $G$ is a connected graph and $S\subseteq V(G)$ with $S\ne \emptyset$, then there exists an induced connected subgraph $H$ of $G$ and a subset $S'\subseteq V(H)$ such that $S\subseteq S'$, $|S'| \le 3|S|$ and $\chi(H\setminus S') \le 2$. 
\end{lem}
\begin{proof}
We proceed by induction on $|S|$. First suppose that $|S|=1$. Let $v\in S$. Then $H=\{v\}$ and $S'=S$ are as desired. So we may assume that $|S|\ge 2$. 

Let $v\in S$ and $S_1 := S\setminus \{v\}$. By induction, there exists an induced connected subgraph $H_1$ of $G$ and $S_1'\subseteq V(G)$ such that $S_1\subseteq S_1' \subseteq V(H_1)$, $|S_1'| \le 3|S_1|$ and $\chi(H_1\setminus S_1') \le 2$.

If $v\in V(H_1)$, then $H_1$ and $S' = S_1'\cup \{v\}$ are as desired. So we may assume that $v\notin V(H_1)$. Let $P$ be a shortest path from $v$ to $V(H_1)$. Let $H := G[V(H_1)\cup V(P)]$. Let $u$ be the end of $P$ in $V(H_1)$ and let $w$ be the neighbor of $u$ in $P$. (Note that $v$ may equal $w$.) Let $S' = S_1' \cup \{u,v,w\}$. 

Now $|S'| \le |S_1'|+3 \le 3|S_1| + 3 = 3|S|$. Moreover, $H$ is an induced connected subgraph of $G$ and $S \subseteq S' \subseteq V(H)$. Let $P':=P\setminus \{u,v,w\}$. Since $P$ is shortest, we have that $P'$ is an induced path in $G$ and hence $\chi(G[V(P')])\le 2$. Furthermore, since $P$ is shortest, no vertex in $P'$ has a neighbor in $V(H_1)$. Hence $\chi(H\setminus S') = \max\{\chi(H_1\setminus S'), \chi(P')\} \le 2$. Thus $H$ and $S'$ are as desired. 
\end{proof}

\subsection{Proof of Lemma~\ref{lem:inseparable2}}\label{s:inseparable2}

We prove a stronger inductive form of Lemma~\ref{lem:inseparable2} as follows. The following two concepts are crucial to the proof of Lemma~\ref{lem:inseparable2}.

\begin{definition}
Let $\mc{A}=\{A_1,\ldots,A_s\}$ be a $K_{s}$ model in a graph $G$. We say a subgraph $H$ of $G$ is \emph{tangent} to $\mc{A}$ if $|V(H)\cap V(A_i)| = 1$ for every $i\in [s]$.

We say a subset $S$ of $V(G)$ is a \emph{core} of $\mc{A}$ if for every distinct $i,j\in [s]$, there exists an edge with one end in $V(A_i)\cap S$ and the other end in $V(A_j)\cap S$.
\end{definition}

We are now ready to state and prove our stronger inductive form of Lemma~\ref{lem:inseparable2}.

\begin{lem}\label{lem:inseparable3}
There exists an integer $C=C_{\ref{lem:inseparable3}} \ge 1$ such that the following holds: Let $t\ge 3$ be an integer and let $T\in[0, \left\lceil \sqrt{\log t} \right\rceil]$ be an integer. Let $G$ be a graph and let 
$$g_{\ref{lem:inseparable3}}(G,t) := \max_{H\subseteq G} \left\{ \frac{\chi(H)}{t}: \vert(H)\le C^2\cdot  t\cdot  \log^4 t,~H \text{ is $K_t$-minor-free }\right\}.$$  If $G$ is $C^2\cdot t\cdot(1+g(G,t))$-chromatic-inseparable and $\chi(G)\ge 2\cdot C^2\cdot t\cdot (1+g(G,t))$, then $G$ contains a $K_t$ minor or $G$ contains both of the following:
\begin{itemize}
\item a $K_{s}$ model $\mc{A}=\{A_1,\ldots,A_s\}$ with a core $S$ where $s = T \cdot \lceil \frac{t}{\sqrt{\log t}}\rceil$ and $|S| \le 2\cdot T^2 \cdot C_{\ref{t:SmallConn2}}^2 \cdot t \cdot \log^3 t$, and
\item a $C \cdot t$-connected subgraph $H$ such that $\chi(H)\ge \chi(G) - \frac{C^2}{2}\cdot t\cdot (1+g(G,t))$, $H$ is tangent to $\mc{A}$, and $V(H)\cap V(\mathcal{A}) \subseteq S$.
\end{itemize}
\end{lem}
\begin{proof}
We show that $C_{\ref{lem:inseparable3}} := 24\cdot C_{\ref{t:SmallConn2}} \cdot C_{\ref{t:linked}} \cdot C_{\ref{t:knitted}} \cdot C_{\ref{l:rooted3}}$ suffices. Suppose not. Let $t$ be an integer where there exists a counterexample for some $T$ and $G$ as in the hypotheses of the lemma. Let $T$ and $G$ be a counterexample for that value $t$ such that $T$ is minimized.  Note that the lemma holds for $T=0$ trivially and hence we may assume $T>0$. 

Let $x := \left\lceil \frac{t}{\sqrt{\log t}} \right\rceil$. Note that $x \le t$ since $t\ge 3$. Let $s':=(T-1)x$ and recall that $s=Tx$. Note that $s'\le (\sqrt{\log t})\left(\frac{t}{\sqrt{\log t}} + 1\right) \le 2t$ since $t\ge 3$ and hence $s = s'+x \le 3t$. 

By the minimality of $T$, there exists a $K_{s'}$ model $\mc{A'}=\{A_1',\ldots,A_{s'}'\}$ with a core $S'$ where $|S'| \le 2\cdot (T-1)^2\cdot  C_{\ref{t:SmallConn2}}^2 \cdot t \cdot \log^3 t$, and a $k$-connected subgraph $H'$ such that $\chi(H')\ge \chi(G) -  \frac{C_{\ref{lem:inseparable3}}^2}{2}\cdot t\cdot (1+g(G,t))$, $H'$ is tangent to $\mc{A}'$, and $V(H')\cap V(\mathcal{A}')\subseteq S'$.

For each $i\in[s']$, let $w_i$ be the unique vertex in $V(H)\cap A_i'$. Let $H_1 := H'\setminus \{w_i:i\in[s']\}$. Note that $\chi(H_1)\ge \chi(H')-2t$.

Let $k:=C_{\ref{lem:inseparable3}}\cdot t$. Thus 
\begin{align*}
\chi(H_1)&\ge \chi(G) - \frac{C_{\ref{lem:inseparable3}}^2}{2}\cdot t\cdot (1+g(G,t)) - 2t \\
&\ge C_{\ref{lem:inseparable3}}^2\cdot t\cdot (1+g(G,t)) \\
&\ge 4 \cdot C_{\ref{t:SmallConn2}}\cdot k \cdot (1 + g_{\ref{c:SmallConn}}(G,t)),
\end{align*}
since $C_{\ref{lem:inseparable3}}\ge 4 \cdot C_{\ref{t:SmallConn2}}$ and $g_{\ref{c:SmallConn}}(G,t)) \le g(G,t)$. Given this and the fact that $H_1$ has no $K_t$ minor as $G$ does not, we have that $t, k$ and $H_1$ satisfy the hypotheses of Corollary~\ref{c:SmallConn}.

Thus by Corollary~\ref{c:SmallConn}, there exist $T$ vertex-disjoint $k$-connected subgraphs $J_1, \ldots J_{T-1}, D$ in $H_1$ such that $\vert(J_i) \le C_{\ref{t:SmallConn2}}^2 \cdot t \cdot \log^3 t$ for each $i\in[T-1]$ and $\vert(D) \le C_{\ref{t:SmallConn2}}^2 \cdot t \cdot \log^3 t$. For each $i\in[T-1]$, let $X_i=\{w_{s'+2(i-1)x+1},\ldots w_{s'+2ix}\}$ be a subset of $V(J_i)$ of size $2x$. Let $W_1 := \{w_i: i\in [3s']\}$.

Now $|W_1|\le 3s' \le 6t$. As $H'$ is $12t$-connected since $C_{\ref{lem:inseparable3}} \ge 12$, it follows from Menger's Theorem that there exists a set $\mc{Q}_1$ of $W_1-V(D)$ paths in $H'$ with $|\mc{Q}_1|=2|W_1|$ which are vertex-disjoint except in $W_1$ and where each vertex in $W_1$ is the end of exactly two paths in $\mc{Q}_1$. We may assume without loss of generality that each path in $\mc{Q}_1$ is induced. Hence $\chi\left(G[\bigcup_{Q\in \mc{Q}_1} V(Q)]\right)\le 24t$.

Let $J:=V(D)\cup \bigcup_{i=1}^{T-1} V(J_i)$. Note then that 
$$|J|\le T \cdot C_{\ref{t:SmallConn2}}^2 \cdot t \cdot \log^3 t \le C_{\ref{lem:inseparable3}}^2 \cdot t \cdot \log^4 t,$$
since $C_{\ref{lem:inseparable3}}\ge 2C_{\ref{t:SmallConn2}}$ and $T\le 2 \cdot \sqrt{\log t}$. By definition of $g(G,t)$, we have that 
$$\chi(G[J]) \le t\cdot g(G,t).$$
Let $H_2 := G[V(H') \setminus (J \cup \bigcup_{Q\in \mc{Q}_1} V(Q)) ]$. Hence 
\begin{align*}
\chi(H_2) &\ge \chi(H') - 24t - t\cdot g(G,t) \\
&\ge \chi(H')- k\cdot (1+ g(G,t)),
\end{align*}
where the last inequality follows since $C_{\ref{lem:inseparable3}}\ge 24$. Note then that $\chi(H_2)\ge 7k$ since $C_{\ref{lem:inseparable3}}\ge 14$. So by Theorem~\ref{t:LargeChi}, there exists a $k$-connected subgraph $H_3$ of $H_2$ with 
\begin{align*}
\chi(H_3) &\ge \chi(H_2) - 6k \\
&\ge \chi(H') - 7k\cdot (1+g(G,t)) \\
&\ge \chi(G) - C_{\ref{lem:inseparable3}}^2 \cdot t\cdot (1+g(G,t)),
\end{align*}
where the last inequality follows since $\chi(H') \ge \chi(G) - \frac{C_{\ref{lem:inseparable3}}^2}{2}\cdot t\cdot (1+g(G,t))$ and $C_{\ref{lem:inseparable3}} \ge 14$. Since $H'$ is $12t$-connected as $C_{\ref{lem:inseparable3}} \ge 12$, it follows that $H'\setminus W_1$ is $2s$-connected. By Menger's Theorem, there exists a set $\mc{Q}_2$ of vertex-disjoint $V(H_3)-D$ paths in $H'\setminus W_1$ such that $|\mc{Q}_2|=2s$. 

Let $a_1,\ldots, a_{s}, b_1,\ldots b_{s}$ be the vertices in $V(H_3)$ that are the ends of the paths in $\mc{Q}_2$. For each $i\in [s]$, let $C_i:=\{a_i,b_i\}$. Let $H''$ be obtained from $\bigcup_{Q\in \mathcal{Q}_1\cup \mathcal{Q}_2} Q$ by for each $i\in [s]$, identifying the vertices in $C_i$ to a new vertex denoted $c_i$. Let $C:= \left\{ c_i : i \in [s]\right\}$.

Now by Lemma~\ref{lem:MengerVariant} given the existence of $\mc{Q}_1$ and $\mc{Q}_2$, it follows that there exists a set $\mc{P}'_1$ of vertex-disjoint paths $(W_1\cup C) - V(D)$ paths in $H''$ such that $|\mc{P}'_1|=3s'+s$. 

For each $i \in [3s']$, let $P_i$ be the path in $\mc{P}'_1$ containing $w_i$ and let $u_i$ denote the end of $P_i$ in $D$. For each $i\in [s]$, let $P'_{3s'+i}$ be the path in $\mc{P}'_1$ containing $c_i$ and let $u_{3s'+i}$ denote the end of $P'_{3s'+i}$ in $D$.

For each $i\in [s]$, $P'_{3s'+i}$ corresponds to a $V(H_3)-V(D)$ path $P_{3s'+i}$ in $H'$ from a vertex $w_{3s'+i}$ in $C_i$ to $u_{3s'+i}$. Now $\mc{P}:=\{P_i: i\in [3s'+s]\}$ is a $\left\{(w_i,u_i)\right\} _{i\in [3s'+s]}$-linkage in $H'$ such that $V(\mc{P})\cap V(H_3) = \{w_{3s'+i}: i\in [s]\}$ and $V(\mc{P})\cap V(D) = \{u_i: i\in [3s'+s]\}$. 

Note that $2x \sqrt{\log(2x)} \le 4t$ and $3s'+s\le 10 t$. For $i\in [T-1]$, since $J_i$ is $10\cdot C_{\ref{l:rooted3}}\cdot t$-connected as $C_{\ref{lem:inseparable3}} \ge 10\cdot C_{\ref{l:rooted3}}$, we have by Lemma~\ref{l:rooted3} that $J_i$ is $(2x,3s'+s)$-woven. 

\begin{claim}\label{Linkages1}
There exists a $\left\{(w_i,u_i)\right\}_{ i\in[3s'+s]}$-linkage $\mc{P}'$ with $V(\mc{P}')\subseteq V(\mc{P})$ and $K_{2x}$ models $\mc{M}_i$ in $J_i$ rooted at $X_i$ for $i\in[T-1]$ such that $V(\mc{P}')\cap V(\mc{M}_i) = X_i$.
\end{claim}
\begin{proof}
Since each $J_i$ is $(2x,3s'+s)$-woven, the claim follows by iteratively applying Lemma~\ref{l:woven} to each $J_i$.
\end{proof}

By Claim~\ref{Linkages1}, there exist $\mc{P}'$ and $\mc{M}_i$ for each $i\in [T-1]$ as in the statement of the claim. Now let us define a few parameters. For each $i\in [T-1]$ and $j\in [x]$, define $$a_{i,j}:= s'+2(i-1)x + x + j,$$ 
and let $M_{a_{i,j}}$ denote the subgraph in $\mc{M}_i$ containing $w_{a_{i,j}}$. For each $j\in [s']$, define 
$$b_j := s'+j + x \left(\left\lceil \frac{j}{x} \right\rceil-1\right),$$
$$b'_j := 3s'+j,$$
and let $M_{b_j}$ denote the subgraph in $\mc{M}_{\left\lceil \frac{j}{x} \right\rceil}$ containing $w_{b_j}$. For each $j\in[s']$, let $$S_j := \{u_j, u_{b_j},u_{b'_j}\}.$$ For each $j\in [x]$, let 
$$S_{s'+j} := \{ u_{a_{i,j}} : i\in[T-1]\}\cup \{u_{4s'+j}\}.$$

Note that $D$ is $C_{\ref{t:knitted}}\cdot (3s'+s)$-connected as $3s'+s\le 10t$ and $C_{\ref{lem:inseparable3}}\ge 10\cdot C_{\ref{t:knitted}}$. Hence by Theorem~\ref{t:knitted}, we have that $D$ is $(3s'+s,s)$-knit. By definition then, there exist vertex-disjoint connected subgraphs $D_1,\ldots D_{s}$ of $D$ where $S_j\subseteq D_j$ for each $j\in [s]$.

Now we construct a $K_{s}$ model $\mc{A}''=\{A_1'',\ldots,A_s''\}$ in $G$ with a core $S := S'\cup J\cup \{w_{3s'+i}: i\in [s]\}$ such that $H_3$ is tangent to $\mc{A}''$ as follows:
\begin{itemize}
\item For $j\in [s']$, let $A''_j := A'_j \cup P_j \cup D_j \cup P_{b_j} \cup \mc{M}_{b_j} \cup P_{b'_j}$.
\item For $j\in [x]$, let $A''_{s'+j} := P_{4s'+j} \cup D_{s'+j} \cup \bigcup_{i\in[T-1]} \left( P_{a_{i,j}} \cup \mc{M}_{a_{i,j}} \right)$.
\end{itemize}
Note that $s'+x = s$ and
\begin{align*}
|S| &\le |S'| + T \cdot C_{\ref{t:SmallConn2}}^2 \cdot t \cdot \log^3 t + s\\
&\le 2\cdot (T-1)^2 \cdot C_{\ref{t:SmallConn2}}^2 \cdot t \cdot \log^3 t + T \cdot C_{\ref{t:SmallConn2}}^2 \cdot t \cdot \log^3 t + Tx\\
&\le 2\cdot T^2 \cdot C_{\ref{t:SmallConn2}}^2 \cdot t \cdot \log^3 t.
\end{align*} 
Note though that $\chi(H_3)$ may be smaller than $\chi(G)-\frac{C_{\ref{lem:inseparable3}}^2}{2} \cdot t\cdot (1+g(G,t))$, so we are not yet finished with the proof. However recall that $\chi(H_3) \ge \chi(G) - C_{\ref{lem:inseparable3}}^2\cdot t\cdot (1+g(G,t))$.

For each $i\in [s]$, by Lemma~\ref{lem:Steiner} there exists an induced connected subgraph $A_i$ of $G[V(A_i'')]$ and a set $X_i'\subseteq V(A_i)$ such that $S\cap V(A_i'') \subseteq X_i'$, $|X_i'| \le 3 \cdot |S\cap V(A_i'')|$ and $\chi( G[V(A_i)]\setminus X_i' ) \le 2$.

Now $\mc{A} := \{A_i: i\in [s]\}$ is a $K_{s}$ model in $G$ such that $S$ is a core of $\mc{A}$, $H_3$ is tangent to $\mc{A}$, and $V(H_3)\cap V(\mathcal{A})\subseteq S$. Let $A := \bigcup_{i\in [s]}V(A_i)$. Let $X := \bigcup_{i\in [s]} X_i'$. Then
$$\chi(G[A] \setminus X)\le 2s\le 6t.$$ 
Note that $|X| \le 3|S|$ and hence 
$$|X|\le 6\cdot T^2 \cdot C_{\ref{t:SmallConn2}}^2 \cdot t \cdot \log^3 t \le C_{\ref{lem:inseparable3}}^2 \cdot t \cdot \log^4 t,$$
since $C_{\ref{lem:inseparable3}} \ge 6 \cdot C_{\ref{t:SmallConn2}}$. Thus by definition of $g(G,t)$, we have 
$$\chi(G[X]) \le t\cdot g(G,t),$$
and hence
$$\chi(G[A])\le 6t + t\cdot g(G,t)\le  k\cdot (1+g(G,t))$$
since $C_{\ref{lem:inseparable3}}\ge 6$.

Let $H_4 := G \setminus A$. Thus $\chi(H_4) \ge \chi(G)-k\cdot (1+g(G,t))$. Note then that $\chi(H_4)\ge 7k$ since $C_{\ref{lem:inseparable3}}\ge 7$. So by Theorem~\ref{t:LargeChi}, there exists a $k$-connected subgraph $H_5$ of $H_4$ with 
\begin{align*}
\chi (H_5) &\ge \chi(H_4) - 6k \\
&\ge \chi(G)-k\cdot (1+g(G,t)) - 6k \\
&\ge \chi(G) - \frac{C_{\ref{lem:inseparable3}}^2}{2}\cdot t\cdot (1+g(G,t)),
\end{align*}
since $C_{\ref{lem:inseparable3}}\ge 14$. First suppose that $|V(H_5)\cap V(H_3)| < k$. Then 
\begin{align*}
\chi(H_5\setminus V(H_3)) &\ge \chi(G) - 8k\cdot (1+g(G,t)) \\
&\ge \chi(G) - C_{\ref{lem:inseparable3}}^2\cdot t\cdot (1+g(G,t)),
\end{align*} 
since $C_{\ref{lem:inseparable3}}\ge 8$. Thus $H_5\setminus V(H_3)$ and $H_3$ are vertex-disjoint subgraphs of $G$ each with chromatic number at least $\chi(G) - C_{\ref{lem:inseparable3}}^2\cdot t\cdot (1+g(G,t))$, contradicting that $G$ is $C_{\ref{lem:inseparable3}}^2\cdot t\cdot (1+g(G,t))$-chromatic-inseparable. 

So we may assume that $|V(H_5)\cap V(H_3)| \ge k$. But now $H := H_3\cup H_5$ is a $k$-connected subgraph of $G$. Moreover 
$$\chi(H)\ge \chi(H_5) \ge \chi(G) - \frac{C_{\ref{lem:inseparable3}}^2}{2}\cdot t\cdot (1+g(G,t)).$$ 
Finally, $H$ is tangent to $\mc{A}$ and $V(H)\cap V(\mathcal{A}) \subseteq S$. Hence $\mc{A}$, $S$ and $H$ satisfy the conclusions of the lemma for $T$, a contradiction.
\end{proof}

We are now ready to prove Lemma~\ref{lem:inseparable2} as a direct corollary of Lemma~\ref{lem:inseparable3}. We restate Lemma~\ref{lem:inseparable2} for convenience.

\Inseparable*
\begin{proof}[Proof of Lemma~\ref{lem:inseparable2}]
Follows from Lemma~\ref{lem:inseparable2} with $T=\left\lceil \sqrt{\log t} \right\rceil$ and $C_{\ref{lem:inseparable2}}=C_{\ref{lem:inseparable3}}^2$.
\end{proof}

\section{Finding a Minor in General}\label{s:separable}

In this section, we prove Theorem~\ref{t:tech}. Before proceeding to the proof of Theorem~\ref{t:tech}, we first discuss in the next subsection a technical obstacle and how to overcome it.

\subsection{Proof Overview}

A proof for the ``all high chromatic subgraphs are chromatic-separable'' case essentially follows by creating (using the definition of chromatic-separable and Theorem~\ref{t:LargeChi}) a rooted trinary tree where the root corresponds to the original graph and the children of a vertex correspond to three vertex-disjoint highly-connected high-chromatic subgraphs (because if chromatic-separability yields two such graphs, then separating the second into two more yields three). As long as the number of levels of this tree is at least $\Omega(\log \log t)$, it will be possible to recursively build a $K_t$ minor through these levels by splitting the construction into smaller complete minors (i.e. $K_{2t/3}$ minor for one level down) and using Theorem~\ref{t:linked} to link them back together. This works because on the leaves one can use Lemma~\ref{l:rooted3} to make a complete minor of size $K_{t/\log t}$ in each leaf. More formally, one inductively shows that the graphs at level $i$ are $(a_i,b_i)$-woven where $a_i = t\cdot (2/3)^i$ and $b_i = \sum_{j=0}^{i-1} 3a_j = \sum_{j=0}^{i-1} 3t \cdot (2/3)^j \le 9t$.

The obstacle for proving Theorem~\ref{t:tech} then is that too much chromatic number is lost in each recursion step, i.e. $\Omega(t \cdot g_{\ref{lem:inseparable2}}(G,t))$; since there are $\Omega(\log \log t)$ levels, this requires chromatic number $\Omega(t \cdot g_{\ref{lem:inseparable2}}(G,t) \cdot \log\log t)$ which yields a bound of $O(t (\log \log t)^2)$ for the chromatic number of $K_t$-minor-free graphs. The key problem is that we are using the chromatic-inseparable case as a distinct step at each level; that is, we are finding a new $K_t$ minor at a chromatic-inseparable node instead of finding a $K_{a_i}$ minor and linking it appropriately. 

The first step to overcome this obstacle is to use Lemma~\ref{lem:MinorToWoven} which shows that the existence of a complete minor in a highly-connected graph is enough to allow wovenness (i.e.~rooted complete minors). Unfortunately, this is still not enough to overcome the obstacle. If it sufficed to have a $K_{2a_i}$ minor, then via Lemma~\ref{lem:inseparable2} the chromatic number lost in each step would be $\Omega(a_i \cdot g_{\ref{lem:inseparable2}}(G,a_i))$ and so the total chromatic number lost would be $O\left( \sum_{j=0}^{i-1} a_j \cdot g_{\ref{lem:inseparable2}}(G,a_j)\right)$ which given the definition of $f_{\ref{t:tech}}$ (as using maximum over large enough $a$) would be $O(t \cdot f_{\ref{t:tech}}(G,t))$ as desired. The issue is that we need a node at level $i$ to be $(a_i,b_i)$-woven and so we would need say a $K_{2(a_i+2b_i)}$ minor and yet $b_i$ is on the order of $t$.  

Our solution to this issue is in each step to {\bf link before separating}. More specifically, we find a small dense subgraph (or the desired $K_{2(a_i+2b_i)}$ minor), which incurs a loss in chromatic number of $O( (a_i+b_i) \cdot f_{\ref{t:tech}}(G,a_i+b_i))$. We then use the redundancy version of Menger's Theorem (Lemma~\ref{lem:MengerVariant}) to link the remaining high-chromatic highly-connected subgraph as well as the roots and terminals to the small highly-connected subgraph to be linked there. This is essentially one small step of the large sequential process used in the proof of Lemma~\ref{lem:inseparable2}.

Importantly, we do this step {\bf before invoking the chromatic-inseparable case}, that is we only invoke Lemma~\ref{lem:inseparable2} on the remaining high-chromatic subgraph. Hence the child nodes in our recursive tree only have to preserve the linkage among their siblings instead of the linkages of all of their ancestors; that is, we may take $b_i = 3a_{i-1}$ instead of $\sum_{j=0}^{i-1} 3a_j$, which yields a chromatic loss of $O(a_i \cdot g_{\ref{lem:inseparable2}}(G,a_i))$ in step $i$, and hence a total chromatic loss of $O\left(\sum_{i=0}^{\lfloor \log_3 t\rfloor} a_i \cdot g_{\ref{lem:inseparable2}}(G,a_i)\right) = O(t \cdot f_{\ref{t:tech}}(G,t))$ as desired. 

\subsection{Proof of  Theorem~\ref{t:tech}}

We now prove a stronger more technical form of Theorem~\ref{t:tech} as follows.

\begin{thm}\label{t:rooted2} 
There exists an integer $C=C_{\ref{t:rooted2}} \ge 1$ such that the following holds: Let $t\ge 3$ be an integer. Let $G$ be a graph and let 
$$f_{\ref{t:rooted2}}(G,t) := \max_{H\subseteq G} \left\{ \frac{\chi(H)}{a}:~14t \ge a \ge \frac{t}{\sqrt{\log t}},~\vert(H)\le C a \log^4 a,~H \text{ is $K_a$-minor-free }\right\}.$$ 
Suppose that $t$ is a power of three and $a=\left(\frac{2}{3}\right)^i \cdot t$ is a positive integer for some nonnegative integer $i$. If $G$ is a $Ca$-connected graph with 
$$\chi(G) \ge C(t+276 \cdot a \cdot (1+f_{\ref{t:rooted2}}(G,t)) ),$$ 
then $G$ is $(a,3a)$-woven.
\end{thm}
\begin{proof}
Let $C_{\ref{t:rooted2}} := 35 \cdot C_{\ref{t:SmallConn2}}^2 \cdot C_{\ref{lem:inseparable2}}\cdot C_{\ref{t:linked}} \cdot C_{\ref{l:rooted3}}$.

We proceed by induction on $a$. First suppose that $a \le \frac{t}{\sqrt{\log t}}$. Hence $a \sqrt{\log a} \le t$. Since $C_{\ref{t:rooted2}} \ge 3\cdot C_{\ref{l:rooted3}}$, we have by Lemma~\ref{l:rooted3} that $G$ is $(a,3a)$-woven as desired.

So we assume that $a \ge \frac{t}{\sqrt{\log t}}$. Next suppose that $G$ contains a $K_{14a}$ minor. Since $C_{\ref{t:rooted2}} \ge 7$, we have that $G$ is $7a$-connected. Then by Lemma~\ref{lem:MinorToWoven}, we have that $G$ is $(7a,0)$-woven. Hence $G$ is $(a,3a)$-woven as desired. So we may assume that $G$ is $K_{14a}$-minor-free.

Now we proceed to show that $G$ is $(a,3a)$-woven directly. Let $R=\{r_1,\ldots, r_a\}$, $S=\{s_1,\ldots, s_{3a}\},$ $T=\{t_1,\ldots,t_{3a}\}$ be subsets of $V(G)$ such that if $s_i=t_j$ for some $i,j \in [3a]$, then $i=j$. Since $\delta(G) \ge \kappa(G) \ge 8a = 2|R|+|S|+|T|$ as $C_{\ref{t:rooted2}}\ge 8$, we may assume without loss of generality that $R$, $S$ and $T$ are pairwise disjoint.

For all $i\in [3a]$, let $z_{i} := s_i$. For all $i\in [a]$, let $z_{3a+i} := r_i$. For all $i\in [3a]$, let $z_{4a+i} := t_i$. Let $Z := \{z_i: i\in [7a]\}$. Let $G_1 :=G\setminus Z$. Note that 
$$\chi(G_1)\ge \chi(G)-|Z| \ge \chi(G)-7a \ge C_{\ref{t:rooted2}} \cdot (t+269 \cdot a \cdot (1+f_{\ref{t:rooted2}}(G,t)) ).$$ 

Let $k:= 14 \cdot C_{\ref{t:linked}} \cdot a$. Since $C_{\ref{t:linked}}\ge 1$, we have that $k\ge 14a$. Note that $\d(G_1) \ge \frac{\delta(G_1)}{2} \ge \frac{\delta(G)-|Z|}{2}$. Hence we have that 
$$\d(G_1) \ge \frac{\kappa(G)-7a}{2} \ge C_{\ref{t:SmallConn2}} \cdot k,$$
where we used that $C_{\ref{t:rooted2}}\ge 28 \cdot C_{\ref{t:linked}} \cdot C_{\ref{t:SmallConn2}}+7$. By Theorem~\ref{t:SmallConn2} since $G_1$ is $K_{14a}$-minor-free, there exists a $(14 \cdot C_{\ref{t:linked}}\cdot a)$-connected induced subgraph $H_0$ of $G_1$ with $\vert(H_0)\le C_{\ref{t:SmallConn2}}^2 \cdot 14a \cdot \log^3 (14a)$. Since $C_{\ref{t:rooted2}} \ge C_{\ref{t:SmallConn2}}^2$, we find that
$$\vert(H_0) \le C_{\ref{t:rooted2}} \cdot (14a) \cdot \log^4(14a).$$ 
Since $14a \ge \frac{t}{\sqrt{\log t}}$, we have by definition of $f_{\ref{t:rooted2}}(G,t)$ that 
$$\chi(H_0) \le 14a \cdot f_{\ref{t:rooted2}}(G,t).$$

Since $G$ is $14a$-connected as $C_{\ref{t:rooted2}} \ge 14$, it follows from Menger's Theorem that there exists a set $\mathcal{P}_1$ of $Z-V(H_0)$ paths with $|\mathcal{P}_1|=2|Z|$ which are vertex-disjoint except in $Z$ and where each vertex in $Z$ is the end of exactly two paths in $\mathcal{P}_1$.

We may assume without loss of generality that each path in $\mathcal{P}_1$ is induced. Hence 
$$\chi\left(\bigcup_{P\in \mathcal{P}_1} V(P)\right) \le 2\cdot |\mathcal{P}_1| \le 28a.$$ 
Let $G_2:=G\setminus (Z\cup \bigcup_{P\in\mathcal{P}_1} V(P)\cup V(H_0))$. Thus 
$$\chi(G_2) \ge \chi(G_1)-28a-14a\cdot f_{\ref{t:rooted2}}(G,t) \ge C_{\ref{t:rooted2}} \cdot (t+ 227 \cdot a \cdot (1+f_{\ref{t:rooted2}}(G,t)) ).$$
By Theorem~\ref{t:LargeChi} applied to $G_2$, there exists a $(C_{\ref{t:rooted2}}\cdot a)$-connected subgraph $G_3$ of $G_2$ with 
$$\chi(G_3)\ge \chi(G_2) - 6\cdot C_{\ref{t:rooted2}} \cdot a \ge C_{\ref{t:rooted2}} \cdot (t+221 \cdot a \cdot (1+f_{\ref{t:rooted2}}(G,t)) ).$$ 

First suppose that $G_3$ is $(a,0)$-woven. Since $G_1$ is $2a$-connected, it follows from Menger's Theorem that there exists a set $\mathcal{P}_2$ of pairwise vertex-disjoint $V(G_3)-V(H_0)$ paths in $G_1$ with $|\mathcal{P}_2|=2a$. Given the existence of $\mathcal{P}_1$ and $\mathcal{P}_2$, it follows from Theorem~\ref{lem:MengerVariant} that there exists a set $\mathcal{P}$ of $8a$ pairwise vertex-disjoint $(Z \cup V(G_3))-V(H_0)$ paths where every vertex in $Z$ is an end of some path in $\mathcal{P}$.

For $i\in [7a]$, let $P_i$ be the path in $\mathcal{P}$ such that $z_i$ is an end of $P_i$. Let $\mathcal{P}\setminus \{P_i:i\in [7a]\} = \{P_{7a+1}, \ldots, P_{8a}\}$. For each $i\in \{7a+1,\ldots, 8a\}$, let $v_i$ be the end of $P_i$ in $V(G_3)$. For each $i\in[8a]$, let $u_i \in V(H_0)$ be the other end of $P_i$. By definition of $(a,0)$-woven, there exists a $K_a$ model $\mc{M}'$ in $G_3$ rooted on $\{v_{7a+1},\ldots, v_{8a}\}$. For each $i\in [a]$, let $M_i'$ denote the subgraph in $\mc{M}'$ containing $v_i$.  

Since $H_0$ is $(C_{\ref{t:linked}} \cdot 4a)$-connected, we have by Theorem~\ref{t:linked} that there exists a $\{(u_{i}, u_{4a+i})\}_{i\in[4a]}$-linkage $\mathcal{Q}$ in $H_0$. For each $i\in [a]$, let 
$$M_i := P_{3a+i}\cup Q_{3a+i}\cup P_{7a+i} \cup M_i'.$$ 
Then $\bigcup_{i\in [a]} M_i$ is $K_a$ model $\mc{M}$ rooted at $R$. For each $i\in [3a]$, let 
$$P_i' := P_{i}\cup Q_{i} \cup P_{4a+i}.$$ 
Now $\bigcup_{i\in [3a]} P_i'$ is an $S-T$ linkage disjoint from $\mc{M}$. Since $R,S,T$ were arbitrary, we find that $G$ is $(a,3a)$-woven as desired.

Thus it suffices to show that $G_3$ is $(a,0)$-woven. To that end, let $R'=\{r_1',\ldots, r_a'\}$ be a set of $a$ distinct vertices in $G_3$. 

Since $G_3$ is $(C_{\ref{t:rooted2}}\cdot a)$-connected, we have that $\delta(G_3) \ge C_{\ref{t:rooted2}}\cdot a \ge 3a = 3|R'|$ since $C_{\ref{t:rooted2}}\ge 3$. Thus there exists a subset $S'=\{s_1',\ldots, s_{2a}'\}$ of $V(G_3)\setminus R'$ with $|S'|=2a$, where $s_i'$ and $s_{a+i}'$ are adjacent to $r_i'$ for every $i\in [a]$. Let $G_4 := G_3\setminus (R'\cup S')$. So 
$$\chi(G_4)\ge \chi(G_3)-3a \ge C_{\ref{t:rooted2}} \cdot (t+218 \cdot a \cdot (1+f_{\ref{t:rooted2}}(G,t)) ).$$

Let $m:=C_{\ref{lem:inseparable2}} \cdot 14a \cdot (1+g_{\ref{lem:inseparable2}}(G,14a))$. Since $a\ge \frac{t}{\sqrt{\log t}}$ and $C_{\ref{t:rooted2}}\ge C_{\ref{lem:inseparable2}}$, we find that $g_{\ref{lem:inseparable2}}(G,14a) \le f_{\ref{t:rooted2}}(G,t)$. Since $C_{\ref{t:rooted2}}\ge C_{\ref{lem:inseparable2}}$, we find that
$$m \le C_{\ref{t:rooted2}} \cdot 14a \cdot (1+f_{\ref{t:rooted2}}(G,t)).$$
It follows that $\chi(G_4) \ge 2m$. Then since $g_{\ref{lem:inseparable2}}(G,14a) \ge g_{\ref{lem:inseparable2}}(G_4,14a)$ and $G_4$ is $K_{14a}$-minor-free, we have by Lemma~\ref{lem:inseparable2} that $G_4$ is $m$-chromatic-separable. 

Thus by definition of $m$-chromatic-separable, there exist two vertex-disjoint subgraphs $H_1,H_2$ of $G_4$ such that for each $i\in \{1,2\}$, we have
$$\chi(H_i)\ge \chi(G_4)-m \ge C_{\ref{t:rooted2}} \cdot (t+204 \cdot a \cdot (1+f_{\ref{t:rooted2}}(G,t)) ).$$

Similarly since $\chi(H_1) \ge 2m$ and $H_1$ is $K_{14a}$-minor-free, it also follows from Lemma~\ref{lem:inseparable2} that $H_1$ is $m$-chromatic-separable. Hence by definition of $m$-chromatic-separable, there exist two vertex-disjoint subgraphs $H_{1,1}$, $H_{1,2}$ of $H_1$ such that for each $j\in\{1,2\}$, we have
$$\chi(H_{1,j})\ge \chi(H_1)-m \ge C_{\ref{t:rooted2}} \cdot (t+190 \cdot a \cdot (1+f_{\ref{t:rooted2}}(G,t)) ).$$
Let $J_1 := H_{1,1}$, $J_2 := H_{1,2}$ and $J_3 :=H_{2}$. Note that for each $i\in [3]$, $\chi(J_i)\ge 7 \cdot C_{\ref{t:rooted2}} \cdot a$. Hence for each $i\in [3]$, by Theorem~\ref{t:LargeChi} as $\chi(J_i)\ge 7 \cdot C_{\ref{t:rooted2}} \cdot a$, there exists a $(C_{\ref{t:rooted2}}\cdot a)$-connected subgraph $J'_i$ of $J_i$ with 
$$\chi(J'_i)\ge \chi(J_i) - 6\cdot C_{\ref{t:rooted2}} \cdot a \ge (t+184 \cdot a \cdot (1+f_{\ref{t:rooted2}}(G,t)) ).$$ 
Let $a' := 2a/3$. Thus $276 \cdot a' = 184 \cdot a$. Hence for each $i\in [3]$, we have that  
$$\chi(J'_i)\ge C_{\ref{t:rooted2}}\cdot (t+276\cdot a'\cdot (1+f_{\ref{t:rooted2}}(G,t))).$$
Thus by induction on $a$, we find that $J'_i$ is $(a',3a')$-woven for each $i\in [3]$.

For each $i\in [3]$, since $\vert(J'_i)\ge t$, there exists a subset $T_i = \left\{t_{(i-1)a'+1}', \ldots, t_{ia'}'\right\}$ of $V(J'_i)$. Let $T' :=\{t_i' : i\in [2a]\}$. Note that $|S'|=|T'|=2a$. Since $G_3\setminus R'$ is $(C_{\ref{t:linked}} \cdot (2a))$-connected since $C_{\ref{t:rooted2}} \ge 2 C_{\ref{t:linked}}+1$, we have by Theorem~\ref{t:linked} that there exists an $\left\{(s_i',t_i')\right\}_{i\in [2a]}$-linkage $\mc{P'}$ in $G_3\setminus R'$.

\begin{claim}\label{Linkages2}
There exists an $\left\{(s_i',t_i')\right\}_{i\in [2a]}$-linkage $\mc{P''}$ in $G_3\setminus R'$ and $K_{a'}$ models $\mc{M}_i$ in $J'_i$ rooted at $T_i$ for $i\in [3]$ such that $V(\mc{P''})\cap V(\mc{M}_i) = T_i$.
\end{claim}
\begin{proof}
Since each $J'_i$ is $(a',3a')$-woven, the claim follows by iteratively applying Lemma~\ref{l:woven} to each $J'_i$.
\end{proof}

By Claim~\ref{Linkages2}, there exist $\mc{P''}$ and $K_{a'}$ models $\mc{M}_i$ as in the claim. For each $i\in[2a]$, let $P''_i$ be the path in $\mc{P''}$ containing $\{s_i',t_i'\}$. 

For each $i\in [3]$ and $j\in [a']$, let $M'_{(i-1)a'+j}$ denote the subgraph in $\mc{M}_i$ containing $t'_{(i-1)a'+j}$. For each $i\in [a]$, let 
$$M_i := M'_{i} \cup P''_{i} \cup \{r'_is'_{i}, r'_is'_{a+i}\} \cup P''_{a+i} \cup M'_{a+i}.$$ 
Now $\mc{M} := \{M_i : i\in [a]\}$ is a $K_a$ model in $G_3$ rooted at $R'$. Thus $G_3$ is $(a,0)$-woven as desired.
\end{proof}

We are now ready to prove Theorem~\ref{t:tech}, which we restate for convenience.

\Tech*
\begin{proof}[Proof of Theorem~\ref{t:tech}]
Let $C_{\ref{t:tech}} := 3^9 \cdot C_{\ref{t:rooted2}}$. Suppose not. Let $t' := 3^{\lceil \log_3 t \rceil}$, that is the smallest power of three that is at least as large as $t$. Thus $t'\le 3t$.

By Theorem~\ref{t:LargeChi} with $k=C_{\ref{t:rooted2}}\cdot t'$ as $\chi(G)\ge 21 \cdot C_{\ref{t:rooted2}}\cdot t$, there exists a $(C_{\ref{t:rooted2}}\cdot t')$-connected subgraph $G'$ of $G$ with 
$$\chi(G')\ge \chi(G)-18\cdot C_{\ref{t:rooted2}}\cdot t \ge 828\cdot C_{\ref{t:rooted2}}\cdot t\cdot (1+f_{\ref{t:tech}}(G,t)),$$
since $f_{\ref{t:tech}}(G,t) \ge 0$ and $3^9 \ge 828$. Since $t'\le 3t$, we find that 
$$\chi(G') \ge C_{\ref{t:rooted2}}\cdot t' \cdot (1+276 \cdot f_{\ref{t:tech}}(G,t)).$$

We claim that $f_{\ref{t:tech}}(G,t)\ge f_{\ref{t:rooted2}}(G',t')$. To see this, let $H\subseteq G'$ and $a$ be such that $14t' \ge a \ge \frac{t'}{\sqrt{\log t'}}$, $\vert(H)\le C_{\ref{t:rooted2}} \cdot a \log^4 a$, $H$ is $K_a$-minor-free, and $\frac{\chi (H)}{a} = f_{\ref{t:rooted2}}(G',t')$. First suppose $a\le t$. Then since $C_{\ref{t:rooted2}}\le C_{\ref{t:tech}}$, we find that $\frac{\chi(H)}{a} \le f_{\ref{t:tech}}(G,t)$ as desired. So we assume that $a > t$. Yet $a\le 14t' \le 42t$. Hence 
$$\vert(H)\le C_{\ref{t:rooted2}} \cdot (42t) \log^4 (42t) \le 3^9 \cdot C_{\ref{t:rooted2}} \cdot t \cdot \log^4 t \le C_{\ref{t:tech}} \cdot t \cdot \log^4 t,$$
where we used that $42 \log^4 (42t) \le 3^9 \cdot \log t$ as $t\ge 3$ and that $C_{\ref{t:tech}} \ge 3^9 \cdot C_{\ref{t:rooted2}}$. Since $H$ is $K_t$-minor-free as $G$ is, it follows that $\frac{\chi(H)}{a} \le f_{\ref{t:tech}}(G,t)$ as desired. This proves the claim.

Thus
$$\chi(G') \ge C_{\ref{t:rooted2}}\cdot t' \cdot (1+276\cdot f_{\ref{t:rooted2}}(G',t')).$$ 
It now follows from Lemma~\ref{t:rooted2} with $a=t'$ that $G'$ is $(t',3t')$-woven and hence contains a $K_{t'}$ minor, contradicting that $G$ has no $K_t$ minor.
\end{proof}

\section*{Acknowledgments}

We thank Paul Seymour for bringing to our attention that Lemma~\ref{lem:MinorToWoven} was already known in the literature, namely in Kawarabayashi's paper from 2007~\cite{Kaw07}. We also thank Paul Wollan for pointing out that we could derive Lemma~\ref{l:rooted} directly from Theorem~\ref{t:rootedMinors}, which is Theorem 1.1 in his paper from 2008~\cite{W08}. We would also like to thank Tung Nguyen, David Wood, Jofre Costa, and the anonymous referee for helpful comments.

\bibliographystyle{plain}
\bibliography{lpostle}

\end{document}